\documentclass[11 pt]{amsart}

\usepackage{amssymb, mathrsfs, fullpage, mathtools, verbatim} 

\newtheorem{theorem}{Theorem}[section]
\newtheorem{lemma}[theorem]{Lemma}
\newtheorem{proposition}[theorem]{Proposition}

\theoremstyle{definition}
\newtheorem{definition}[theorem]{Definition}

\allowdisplaybreaks[1]

\newcommand{\N}{\ensuremath{\mathbb{N}}}

\DeclarePairedDelimiter{\ceil}{\lceil}{\rceil}

\DeclareMathOperator{\injrad}{inj}
\DeclareMathOperator{\supp}{supp}
\DeclareMathOperator{\vol}{vol}

\newcommand{\x}{\ensuremath{\times}}

\newcommand{\dif}{\ensuremath{\partial}}
\newcommand{\intd}{\ensuremath{\,\textrm{d}}}
\newcommand{\Ob}{\ensuremath{\mathcal{O}}}
\newcommand{\Rm}{\ensuremath{\mathrm{Rm}}}
\newcommand{\Rc}{\ensuremath{\mathrm{Rc}}}
\newcommand{\Obhat}{\ensuremath{\widehat{\mathcal{O}}}}

\renewcommand{\phi}{\varphi}
\begin{document}

\title{Ambient Obstruction Flow}
\author{Christopher Lopez}
\date{\today}

\begin{abstract}
We establish fundamental results for a parabolic flow of Riemannian metrics  introduced by Bahuaud-Helliwell in \cite{BahuaudHelliwellExistence} which is based on the Fefferman-Graham ambient obstruction tensor. First, we obtain local $L^2$ smoothing estimates for the curvature tensor and use them to prove pointwise smoothing estimates for the curvature tensor. We use the pointwise smoothing estimates to show that the curvature must blow up for a finite time singular solution. We also use the pointwise smoothing estimates to prove a compactness theorem for a sequence of solutions with bounded $C^0$ curvature norm and injectivity radius bounded from below at one point. Finally, we use the compactness theorem to obtain a singularity model from a finite time singular solution and to characterize the behavior at infinity of a nonsingular solution.
\end{abstract}

\maketitle

\section{Introduction} \label{sec: introduction}
\subsection{Introduction}
The uniformizatiom theorem ensures that for a compact two dimensional Riemannian manifold $(M,g)$, there is a metric $\tilde{g}$ conformal to $g$ for which $(M,\tilde{g})$ has constant sectional curvature equal to $K$. Moreover, the sign of $K$ can be determined via the Gauss-Bonnet theorem. In higher dimensions, curvature functionals have been used with great success to define and locate optimal metrics in higher dimensions; see \cite{LeBrunCurvatureFunctionalsOptimalMetricsDifferentialTopology4Mfds}. One conformally invariant curvature functional for a $4$-dimensional Riemannian manifold $(M,g)$ is given by
\begin{equation*}
\mathcal{F}_W^4(g)=\int_M|W_g|^2\,dV_g,
\end{equation*}
where $W_{ijkl}$ is the Weyl tensor. The negative gradient of $\mathcal{F}_W^4$ is the Bach tensor $B_{ij}$ defined as
\begin{equation*}
B_{ij}=-\nabla^k\nabla^lW_{kijl}-\tfrac{1}{2}R^{kl}W_{kijl}.
\end{equation*}
The study of critical metrics for $\mathcal{F}_W^4$, ie. Bach-flat metrics, has been fruitful. The class of Bach-flat metrics contains, as  shown in \cite{BesseEinsteinManifolds}, familiar metrics such as locally conformally Einstein metrics, scalar flat (anti) self-dual metrics. 

Another conformally invariant functional for a $4$-dimensional Riemannian manifold $(M,g)$ is given by
\begin{equation*}
\mathcal{F}_Q^4(g)=\int_{M} Q(g)\,dV_g,
\end{equation*}
where $Q(g)$ is a scalar quantity introduced by Branson in \cite{BransonFunctionalDeterminant} called the $Q$ curvature.
Via the Chern-Gauss-Bonnet theorem, this functional is related to $\mathcal{F}_W^4$ by $\mathcal{F}_Q^4=8\pi^2\chi(M)-\frac{1}{4}\mathcal{F}_W^4$. The Bach tensor is also the gradient of $\mathcal{F}_Q^4$. Unlike the Weyl tensor, the $Q$ curvature is not pointwise conformally covariant.

One can generalize the $Q$ curvature to a scalar quantitiy defined on $n$ dimensional Riemannian manifolds $(M,g)$, where $n$ is even. Consider the functionals defined for $n$ even by
\begin{equation*}
\mathcal{F}_Q^n(g)=\int_M Q(g)\,dV_g.
\end{equation*}
These functionals are conformally invariant. The gradient of $\mathcal{F}_Q^n$ is a symmetric $2$-tensor $\Ob$, introduced by Fefferman and Graham in \cite{FeffermanGrahamConformalInvariants}, called the ambient obstruction tensor. This tensor arises in physics: for example, Anderson and Chru\'{s}ciel use $\Ob$ in \cite{AndersonChruscielAsymptoticallySimpleSolutionsOfTheVacuumEinsteinEquationsInEvenDimensions} to construct global solutions of the vacuum Einstein equation in even dimensions. In dimension $4$, $\Ob$ is just the Bach tensor. The ambient obstruction tensor is conformally covariant in $n$ dimensions. This is in contrast to the $n$ dimensional generalization of the Bach tensor, which is only conformally covariant in dimension $4$. This fact follows from a result in Graham-Hirachi \cite{GrahamHirachiAmbientOstructionTensorQ_Curvature} stating that in even dimensions 6 and greater, the only conformally covariant tensors essentially are $W$ and $\Ob$. Extending the $4$-dimensional case, Fefferman and Graham showed in \cite{FeffermanGrahamAmbientMetric} that $\Ob$ vanishes for Einstein metrics for all even dimensions. However, there also exist non conformally Einstein metrics for which $\Ob=0$, as shown by Gover and Leitner in \cite{GoverLeitnerSubproductConstructionOfPoincareEinsteinMetrics}. The conformal covariance of $\Ob$ and the fact that obstruction flat metrics generalize conformally Einstein metrics suggest that studying the critical points of $\mathcal{F}_Q^n$ via its gradient flow may aid in the study of optimal metrics on $M$. Our main goal is to establish fundamental results for this gradient flow.

\subsection{Main Results}
We will continue the study of a variant of the gradient flow of $\mathcal{F}_Q^n$, that was introduced by Bahuaud and Helliwell in \cite{BahuaudHelliwellExistence}, establishing fundamental results. This flow, which we will refer to as the \textbf{ambient obstruction flow} (AOF), is defined for a family of metrics $g(t)$ on a smooth manifold $M$ by
\begin{equation}
 \begin{cases}
  \partial_t g=(-1)^\frac{n}{2}\Ob+\frac{(-1)^\frac{n}{2}}{2(n-1)(n-2)}(\Delta^{\frac{n}{2}-1}R)g \label{eqn: Obhat flow intro} \\
  g(0)=h.
 \end{cases}
\end{equation}
The conformal term involving the scalar curvature was added in order to counteract the invariance of $\Ob$ under the action of the conformal group on the space of metrics on $M$. In the papers \cite{BahuaudHelliwellExistence}, \cite{BahuaudHelliwellUniqueness} they proved the short time existence and uniqueness, respectively, of solutions to AOF given by \eqref{eqn: Obhat flow intro}. Kotschwar recently has given in \cite{KotschwarEnergyApproachtoUniquenessForHigherOrderGeometricFlows} an alternate uniqueness proof via a classical energy argument without using the DeTurck trick.

Gradient flows have been studied extensively since Hamilton in \cite{Hamilton3MfdPositiveRicciCurvature,HamiltonCptnessPropertyForSolutionsOfTheRicciFlow,HamiltonFormationOfSingularitiesInTheRicciFlow} and Perelman in \cite{PerelmanEntropyFormulaForTheRicciFlowAndItsGeometricApplications,PerelmanRicciFlowWithSurgeryOnThreeManifolds,PerelmanFiniteExtinctionTime} (expositions are given in \cite{CaoZhuCompleteProofOfThePoincareAndGeometrizationConjectures,KleinerLottNotesOnPerelmanPapers,MorganTianRicciFlowPoincareConjecture}) used the Ricci flow to study the geometry of 3-manifolds. In the past fifteen years, these have begun to include higher order flows. Mantegazza studied a family of higher order mean curvature flows in \cite{MantegazzaSmoothGeometricEvolutionsOfHypersurfaces}, Kuwert-Sch\"{a}tzle studied the gradient flow of the Willmore functional in \cite{KuwertSchaetzleGradientFlowWillmoreFunctional}, Streets studied the gradient flow of $\int_M|\Rm|^2$  in \cite{StreetsGradientFlowIntegral_L2Curvature}, Chen-He studied the Calabi flow in \cite{ChenHeOnTheCalabiFlow,ChenHeCalabiFlowOnKahlerSurfacesWithBoundedSobolevConstant}, and Kisisel-Sarioglu studied the Cotton flow in \cite{KisiselCottonFlow}. Bour studied the gradient flows of certain quadratic curvature functionals in \cite{BourFourthOrderCurvatureFlowsAndGeometricApplications}, including some variants of $\int_M|W|^2$.

Our first result gives pointwise smoothing estimates for the $C^0$ norms of the derivatives of the curvature. Since the AOF PDE \eqref{eqn: Obhat flow intro} is of order $n$, the maximum principle cannot be used to obtain these estimates. Instead, we first use interpolation inequalities derived by Kuwert and Sch\"{a}tzle in \cite{KuwertSchaetzleGradientFlowWillmoreFunctional} in order to derive local integral Bernstein-Bando-Shi-type smoothing estimates. Then, we use a blowup argument adapted from Streets \cite{StreetsLongTimeBehaviorOfFourthOrderCurvatureFlows} in order to convert the integral smoothing estimates to pointwise smoothing estimates, as stated in the following theorem. During the proof, we use the local integral smoothing estimates to take a local subsequential limit of the renormalized metrics.
\begin{theorem}  \label{thm:C^m smoothing estimate along AOF}
 Let $m\geq 0$ and $n\geq 4$. There exists a constant $C=C(m,n)$ so that if $(M^n,g(t))$ is a complete solution to AOF on $[0,T]$ satisfying
 \begin{equation*}
  \max\left(1,\sup_{M\x[0,T]}|\Rm|\right)\leq K,
 \end{equation*}
 then for all $t\in(0,T]$,
 \begin{equation*} \label{eqn:C^m smoothing estimate along AOF:smoothing inequality}
  \sup_M|\nabla^m\Rm|_{g(t)}\leq C\left(K+\frac{1}{t^\frac{2}{n}}\right)^{1+\frac{m}{2}}.
 \end{equation*}
\end{theorem}
We obtain from the pointwise smooting estimates two additional theorems. The first theorem gives an obstruction to the long time existence of the flow. Since the pointwise smoothing estimates do not require that the Sobolev constant be bounded on $[0,T)$, we rule out that the manifold collapses with bounded curvature.
\begin{theorem} \label{thm: obstruction to long time existence}
 Let $g(t)$ be a solution to the AOF on a compact manifold $M$ that exists on a maximal time interval $[0,T)$ with $0<T\leq\infty$. If $T<\infty$, then we must have
 \begin{equation*}
  \limsup_{t\uparrow T}\|\Rm\|_{C^0(g(t))}=\infty.
 \end{equation*}
\end{theorem}
The second theorem allows us to extract convergent subsequences from a sequence of solutions to AOF with uniform $C^0$ curvature bound and uniform injectivity radius lower bound. We prove this in section \ref{sec: compactness of solutions} by using the Cheeger-Gromov compactness theorem to obtain subsequential convergence of solutions at one time. Then, after extending estimates on the covariant derivatives of the metrics from one time to the entire time interval, we obtain subseqential convergence over the entire time interval.
\begin{theorem} \label{thm: cptness for seq of cpt mfds}
 Let $\{(M_k^n,g_k(t),O_k)\}_{k\in\N}$ be a sequence of complete pointed solutions to AOF for $t\in(\alpha,\omega)$, with $t_0\in(\alpha,\omega)$, such that
 \begin{enumerate}
  \item $|\Rm(g_k)|_{g_k} \leq C_0$ on $M_k\times(\alpha,\omega)$ for some constant $C_0<\infty$ independent of $k$
  \item $\injrad_{g_k(t_0)}(O_k)\geq\iota_0$ for some constant $\iota_0>0$.
 \end{enumerate}
  Then there exists a subsequence $\{j_k\}_{k\in\N}$ such that $\{(M_{j_k},g_{j_k}(t),O_{j_k})\}_{k\in\N}$ converges  in the sense of families of pointed Riemannian manifolds to a complete pointed solution to AOF $(M_\infty^n,g_\infty(t),O_\infty)$ defined for $t\in(\alpha,\omega)$ as $k\to\infty$.
\end{theorem}

We use this compactness theorem to prove two corollaries. For a compact Riemannian manifold $(M,g)$, let $C_S(M,g)$ denote the $L^2$ Sobolev constant of $(M,g)$, defined as the smallest constant $C_S$ such that
\begin{equation*}
\|f\|_{L^{\frac{2n}{n-2}}}^2\leq C_S\big(\|\nabla f\|_{L^2}^2+V^{-\frac{2}{n}}\|f\|_{L^2}^2\big),
\end{equation*}
where $V=\vol(M,g)$. The following result states that if the Sobolev constant and the integral of $Q$-curvature are bounded along the flow, there exists a sequence of renormalized solutions to AOF that converge to a singularity model.  
\begin{theorem}  \label{thm: existence of singularity model at curvature singularity}
 Let $(M^n,g(t))$, $n\geq 4$, be a compact solution to AOF that exists on a maximal time interval $[0,T)$. Suppose that $\sup\{C_S(M,g(t)):t\in[0,T)\}<\infty$. Let $\{(x_i,t_i)\}_{i\in\N}\subset M\x[0,T)$ be a sequence of points satisfying $t_i\to T$, $|\Rm(x_i,t_i)|=\sup\{|\Rm(x,t)|:(x,t)\in M\x[0,t_i]\}$, and $\lambda_i\to\infty$, where $\lambda_i=|\Rm(x_i,t_i)|$. Then the sequence of pointed solutions to AOF given by $\{(M,g_i(t),x_i)\}_{i\in\N}$, with
 \begin{equation*}
  g_i(t)=\lambda_i g(t_i+\lambda_i^{-\frac{n}{2}}t),\quad t\in[-\lambda_i^{\frac{n}{2}}t_i,0]
 \end{equation*}
 subsequentially converges in the sense of families of pointed Riemannian manifolds to a nonflat, noncompact complete pointed solution $(M_\infty,g_\infty(t),x_\infty)$ to AOF defined for $t\in(-\infty,0]$. Moreover, if $n=4$ or
\begin{equation*}
\sup_{t\in[0,T)}\int_M Q(g(t))\,dV_{g(t)}<\infty,
\end{equation*}
then $\Ob(g_\infty(t))\equiv 0$ for all $t\in(-\infty,0]$.
\end{theorem}

The next result states that if a nonsingular solution to AOF does not collapse at time $\infty$ and the integral of $Q$-curvature is bounded along the flow, there exists a sequence of times $t_i\to\infty$ for which $g(t_i)$ converges to an obstruction flat metric. We note that in cases (2) and (3), the boundedness of the integral of the $Q$ curvature along the flow implies that $g_\infty(t)$ is obstruction flat. However, this does not imply that $\dif_t g_\infty=0$. Rather, $\dif_t g_\infty=(-1)^{n/2}C(n)(\Delta^{\frac{n}{2}-1}R)g$, i.e. the metric is still flowing by the conformal term of AOF within the conformal class of $g_\infty(0)$.
\begin{theorem} \label{thm: nonsingular limit at infinity}
Let $(M,g(t))$ be a compact solution to AOF on $[0,\infty)$ such that
\begin{equation*}
\sup_{t\in[0,\infty)}\|\Rm\|_{C^0(g(t))}<\infty.
\end{equation*}
Then exactly one of the following is true:
\begin{enumerate}
\item $M$ collapses when $t=\infty$, i.e.
\begin{equation*}
\lim_{t\to\infty}\inf_{x\in M}\injrad_{g(t)}(x)=0.
\end{equation*}
\item There exists a sequence $\{(x_i,t_i)\}_{i\in\N}\subset M\x[0,\infty)$ such that the sequence of pointed solutions to AOF given by $\{(M,g_i(t),x_i)\}_{i\in\N}$, with
 \begin{equation*}
  g_i(t)=g(t_i+t),\quad t\in[-t_i,\infty)
 \end{equation*}
 subsequentially converges in the sense of pointed Riemannian manifolds to a complete noncompact finite volume pointed solution $(M_\infty,g_\infty(t),x_\infty)$ to AOF defined for $t\in(-\infty,\infty)$.  If $n=4$ or
\begin{equation*}
\sup_{t\in[0,\infty)}\int_M Q(g(t))\,dV_{g(t)}<\infty,
\end{equation*}
then $g_\infty(t)$ is obstruction flat for all $t\in(-\infty,\infty)$.
\item There exists a sequence $\{(x_i,t_i)\}_{i\in\N}\subset M\x[0,\infty)$ such that the sequence of pointed solutions to AOF given by $\{(M,g_i(t),x_i)\}_{i\in\N}$, with
 \begin{equation*}
  g_i(t)=g(t_i+t),\quad t\in[-t_i,\infty)
 \end{equation*}
 subsequentially converges in the sense of pointed Riemannian manifolds to a compact pointed solution $(M_\infty,g_\infty(t),x_\infty)$ to AOF defined for $t\in(-\infty,\infty)$, where $M_\infty$ is diffeomorphic to $M$. If $n=4$ or
\begin{equation*}
\sup_{t\in[0,\infty)}\int_M Q(g(t))\,dV_{g(t)}<\infty,
\end{equation*}
then $g_\infty(t)$ is obstruction flat for all $t\in(-\infty,\infty)$ and there exists a family of metrics $\hat{g}_\infty(t)$ conformal to $g_\infty(t)$ for all $t\in(-\infty,\infty)$, with $\hat{g}_\infty(t)=\hat{g}_\infty(0)$ for all $t\in(-\infty,\infty)$, such that $\hat{g}_\infty(0)$ is obstruction flat and has constant scalar curvatue.
\end{enumerate}
\end{theorem}

\section{Background} \label{sec: background}
\subsection{$Q$ Curvature}
Here we recall a description of $Q$ curvature given by Chang et al. in \cite{ChangWhatIsQCurvature}. The $Q$ curvature was introduced in 4 dimensions by Riegert in \cite{RiegertNonlocalActionForTheTraceAnomaly} and Branson-\O rsted in \cite{BransonOrstedExplicitFunctionalDeterminantsInFourDimensions} and in even dimensions by Branson in \cite{BransonFunctionalDeterminant}. It is a scalar quantity defined on an even dimensional Riemannian manifold $(M^n,g)$. If $n=2$, we define $Q$ to be $Q=-\tfrac{1}{2}R=-K$, where $K$ is the Gaussian curvature of $M$. The Gauss-Bonnet theorem gives $\int Q\, dV=-2\pi\chi(M)$. The $Q$ curvature of a metric $\tilde{g}=e^{2f}g$ is given by $e^{2f}\widetilde{Q}=Q+\mathscr{P}f$, where the Paneitz operator $\mathscr{P}$ introduced by Graham-Jenne-Mason-Sparling in \cite{GrahamEtAlConformallyInvariantPowersOfTheLaplacianI_Existence} is given by $\mathscr{P}f=\Delta f$. If $n=4$, we define $Q$ to be
\begin{equation*}
Q=-\tfrac{1}{6}\Delta R-\tfrac{1}{2}R^{ab}R_{ab}+\tfrac{1}{6}R^2.
\end{equation*}
The Chern-Gauss-Bonnet theorem gives
\begin{equation*}
\int Q\,dV=8\pi^2\chi(M)-\tfrac{1}{4}\int |W|^2.
\end{equation*}
In particular, if $M$ is conformally flat, then $\int Q\,dV=8\pi^2\chi(M)$. The $Q$ curvature of a metric $\tilde{g}=e^{2f}g$ is given by $e^{4f}\widetilde{Q}=Q+\mathscr{P}f$, where the Paneitz operator $\mathscr{P}$ is given by
\begin{equation*}
\mathscr{P}f=\nabla_a[\nabla^a\nabla^b+2R^{ab}-\tfrac{2}{3}Rg^{ab}]\nabla_b f.
\end{equation*}
In general when $n$ is even, we are only able to write down the highest order terms of $Q$ and $\mathscr{P}$:
\begin{gather*}
Q=-\tfrac{1}{n-2}\Delta^{\frac{n}{2}-1}R+\mathrm{lots},\quad \mathscr{P}f=\Delta^{\frac{n}{2}}f+\mathrm{lots}.
\end{gather*}
Nonetheless, $Q$ still has nice conformal properties. Under a conformal change of metric $\tilde{g}=e^{2f}g$, we have $e^{nf}\widetilde{Q}=Q+\mathscr{P}f$. The integral of $Q$ is conformally invariant. In particular, if $M$ is locally conformally flat, we have an analogue of the Gauss-Bonnet theorem:
\begin{equation*}
\int Q\,dV=(-1)^{\frac{n}{2}}(\tfrac{n}{2}-1)!\,2^{n-1}\pi^{\frac{n}{2}}\chi(M).
\end{equation*}

\subsection{Ambient Obstruction Tensor}
Fefferman and Graham proposed in \cite{FeffermanGrahamConformalInvariants} a method to determine the conformal invariants of a manifold from the pseudo-Riemannian invariants of an ambient space it is embedded into. They introduced the \textbf{ambient obstruction tensor} $\Ob$ as an obstruction to such an embedding. They subsequently provided a detailed description of the properties of $\Ob$ in their monograph \cite{FeffermanGrahamAmbientMetric}.

We define several tensors that we will use to express $\Ob$. The Schouten tensor $\mathsf{A}$, Cotton tensor $C$, and Bach tensor $B$ are defined as
\begin{equation*}
\mathsf{A}_{ij}=\tfrac{1}{n-2}\big(R_{ij}-\tfrac{1}{2(n-1)}Rg_{ij}\big),\quad C_{ijk}=\nabla_k\mathsf{A}_{ij}-\nabla_j\mathsf{A}_{ik},\quad B_{ij}=\nabla^k C_{ijk}-\mathsf{A}^{kl}W_{kijl}.
\end{equation*}
We obtain via the identity $\nabla^l\nabla^kW_{kijl}=(3-n)\nabla^kC_{ijk}$ that
\begin{equation*}
B_{ij}=\tfrac{1}{3-n}\nabla^l\nabla^kW_{kijl}+\tfrac{1}{2-n}R^{kl}W_{kijl}.
\end{equation*}
We define the notation $P_k^m(A)$ for a tensor $A$ by
\begin{equation*}
P_k^m(A)=\sum_{i_1+\cdots+i_k=m}\nabla^{i_1}A*\cdots*\nabla^{i_k}A.
\end{equation*}
The following result describes $\Ob$. The form of the lower order terms is implied by the proofs.
\begin{theorem}
(Fefferman-Graham \cite{FeffermanGrahamAmbientMetric}, Theorem 3.8; Graham-Hirachi \cite{GrahamHirachiAmbientOstructionTensorQ_Curvature}, Theorem 2.1) Let $n\geq 4$ be even. The obstruction tensor $\Ob_{ij}$ of $g$ is independent of the choice of ambient metric $\tilde{g}$ and has the following properties:
\begin{enumerate}
\item $\Ob$ is a natural tensor invariant of the metric $g$; ie. in local coordinates the components of $\Ob$ are given by universal polynomials in the components of $g$, $g^{-1}$, and the curvature tensor of $g$ and its covariant derivatives, and can be written just in terms of the Ricci curvature and its covariant derivatives. The expression for $\Ob_{ij}$ takes the form
\begin{align*}
\Ob_{ij} & =\Delta^{n/2-2}(\Delta\mathsf{A}_{ij}-\nabla_j\nabla_i{\mathsf{A}_k}^k)+\sum_{j=2}^{n/2}P_j^{n-2j}(\Rm) \\
& =\frac{1}{3-n}\Delta^{n/2-2}\nabla^l\nabla^kW_{kijl}+\sum_{j=2}^{n/2}P_j^{n-2j}(\Rm),
\end{align*}
where $\Delta=\nabla^i\nabla_i$ and $\mathrm{lots}$ denotes quadratic and higher terms in curvature involving fewer derivatives.
\item One has ${\Ob_i}^i=0$ and $\nabla^j\Ob_{ij}=0$.
\item $\Ob_{ij}$ is conformally invariant of weight $2-n$; ie. if $0<\Omega\in C^\infty(M)$ and $\hat{g}_{ij}=\Omega^2g_{ij}$, then $\hat{\Ob}_{ij}=\Omega^{2-n}\Ob_{ij}$.
\item If $g_{ij}$ is conformal to an Einstein metric then $\Ob_{ij}=0$.
\end{enumerate}
\end{theorem}

C.R. Graham and K. Hirachi express the gradient of $Q$ in terms of $\Ob$:
\begin{theorem}
(\cite{GrahamHirachiAmbientOstructionTensorQ_Curvature}, Theorem 1.1) If $g(t)$ is a 1-parameter family of metrics on a compact manifold $M$ of even dimension $n\geq 4$ and $h=\partial_t|_{t=0}\, g(t)$, then
\begin{equation*}
\left.\frac{\partial}{\partial t}\right|_{t=0}\int_M Q(g(t))\,dV_{g(t)}=(-1)^\frac{n}{2}\frac{n-2}{2}\int_M\big\langle\Ob(g(0)),h\big\rangle\,dV_{g(0)}.
\end{equation*}
\end{theorem}

Define the adjusted ambient obstruction tensor $\Obhat$ to be
\begin{equation} \label{eq: gradient expressed with covariant derivatives of Rm}
\Obhat=(-1)^\frac{n}{2}\Ob+\frac{(-1)^\frac{n}{2}}{2(n-1)(n-2)}(\Delta^{\frac{n}{2}-1}R)g.
\end{equation}

We rewrite $\Obhat$ in terms of the Ricci and scalar curvatures.
\begin{proposition} \label{prop: Obhat expanded}
If $(M,g)$ is a Riemannian manifold, then
\begin{gather}
\Ob=\Delta^{\frac{n}{2}-1}\mathsf{A}-\frac{1}{2(n-1)}\Delta^{\frac{n}{2}-2}\nabla^2 R+\sum_{j=2}^{n/2}P_j^{n-2j}(\Rm) \label{eq: Ob in terms of A} \\
\Obhat=\frac{(-1)^\frac{n}{2}}{n-2}\Delta^{\frac{n}{2}-1}\Rc+\frac{(-1)^{\frac{n}{2}-1}}{2(n-1)}\Delta^{\frac{n}{2}-2}\nabla^2 R+\sum_{j=2}^{n/2}P_j^{n-2j}(\Rm). \nonumber
\end{gather}
\end{proposition}
\begin{proof}
First, we reexpress $\Ob$:
\begin{align*}
{\mathsf{A}_k}^k & =\tfrac{1}{n-2}\left[g^{jk}R_{kj}-\tfrac{1}{2(n-1)}Rg^{jk}g_{kj}\right] \\
& =\tfrac{1}{n-2}\left[R-\tfrac{n}{2(n-1)}R\right] \\
& =\tfrac{1}{2(n-1)}R
\end{align*}
and
\begin{align*}
\Ob_{ij} & =\Delta^{\frac{n}{2}-2}(\Delta\mathsf{A}_{ij}-\nabla_j\nabla_i{\mathsf{A}_k}^k)+\sum_{j=2}^{n/2}P_j^{n-2j}(\Rm) \\
& =\Delta^{\frac{n}{2}-1}\mathsf{A}_{ij}-\frac{1}{2(n-1)}\Delta^{\frac{n}{2}-2}\nabla_j\nabla_i R+\sum_{j=2}^{n/2}P_j^{n-2j}(\Rm).
\end{align*}
Next, we reexpress $\Obhat$ using \eqref{eq: Ob in terms of A}:
\begin{align*}
\Obhat ={} &  (-1)^\frac{n}{2}\Ob+\frac{(-1)^\frac{n}{2}}{2(n-1)(n-2)}(\Delta^{\frac{n}{2}-1}R)g \\
={} & (-1)^\frac{n}{2}\Delta^{\frac{n}{2}-1}\mathsf{A}+\frac{(-1)^{\frac{n}{2}-1}}{2(n-1)}\Delta^{\frac{n}{2}-2}\nabla^2 R+\sum_{j=2}^{n/2}P_j^{n-2j}(\Rm)+\frac{(-1)^\frac{n}{2}}{2(n-1)(n-2)}(\Delta^{\frac{n}{2}-1}R)g \\
\begin{split} 
={} & \frac{(-1)^\frac{n}{2}}{n-2}\Delta^{\frac{n}{2}-1}\Rc+\frac{(-1)^{\frac{n}{2}-1}}{2(n-1)(n-2)}(\Delta^{\frac{n}{2}-1}R)g+\frac{(-1)^{\frac{n}{2}-1}}{2(n-1)}\Delta^{\frac{n}{2}-2}\nabla^2 R \\
& {}+\frac{(-1)^\frac{n}{2}}{2(n-1)(n-2)}(\Delta^{\frac{n}{2}-1}R)g+\sum_{j=2}^{n/2}P_j^{n-2j}(\Rm)
\end{split} \\
={} & \frac{(-1)^\frac{n}{2}}{n-2}\Delta^{\frac{n}{2}-1}\Rc+\frac{(-1)^{\frac{n}{2}-1}}{2(n-1)}\Delta^{\frac{n}{2}-2}\nabla^2 R+\sum_{j=2}^{n/2}P_j^{n-2j}(\Rm).
\end{align*}
\end{proof}

\section{Short Time Existence and Uniqueness} \label{sec: short time existence and uniqueness}
In this section, we derive the evolution equations for the covariant derivatives of the curvature tensor. We then give a theorem asserting the short time existence and uniqueness of solutions to AOF.
% \vspace{-2.0em}
\subsection{Preliminaries}
We collect some facts about Riemannian manifolds that will be used to derive the evolution equations.
\begin{lemma} \label{lem: Delta Rc Bianchi identity}
 (Hamilton \cite{Hamilton3MfdPositiveRicciCurvature}, Lemma 7.2) On any Riemannian manifold, the following identity holds:
 \begin{equation*}
  \Delta R_{jklm}=\nabla_j\nabla_mR_{lk}-\nabla_j\nabla_lR_{mk}+\nabla_k\nabla_lR_{mj}-\nabla_k\nabla_mR_{lj}+\Rm^{*2}.
 \end{equation*}
\end{lemma}
\begin{proposition} \label{prop: grad and Laplacian commutator}
 If $A$ is a tensor on a Riemannian manifold and $k,l\geq 1$, then
 \begin{equation*}
  \nabla^k\Delta^l A=\Delta^l\nabla^k A+\sum_{i=0}^{2l+k-2}\nabla^{2l+k-2-i}\Rm*\nabla^i A.
 \end{equation*}
\end{proposition}
\begin{proof}
 First we claim that $\nabla\Delta^l A=\Delta^l\nabla A+\sum_{i=0}^l\nabla^{2l-1-i}\Rm*\nabla^i A$. For any tensor $A$,
 \begin{align*}
  \nabla\Delta A
  & =\nabla_i\nabla^j\nabla_j A \\
  & =\nabla^j\nabla_i\nabla_j A+\Rm*\nabla A \\
  & =\nabla^j\nabla_j\nabla_i A+\nabla\Rm*A+\Rm*\nabla A \\
  & =\Delta\nabla A+\nabla\Rm*A+\Rm*\nabla A.
 \end{align*}
 Suppose the claim is true for $l-1$. Then
 \begin{align*}
  \nabla^2\sum_{i=0}^{2l-3}\nabla^{2l-3-i}\Rm*\nabla^i A)
   & =\nabla\sum_{i=0}^{2l-3}(\nabla^{2l-2-i}\Rm+\nabla^i A+\nabla^{2l-3-i}\Rm*\nabla^{i+1}A) \\
   & =\nabla\left(\sum_{i=0}^{2l-3}\nabla^{2l-2-i}\Rm*\nabla^i A+\Rm*\nabla^{2l-2}A\right) \\
   & =\begin{aligned}[t]
       \sum_{i=0}^{2l-3}(\nabla^{2l-1-i}\Rm*\nabla^i A+\nabla^{2l-2-i}\Rm+\nabla^{i+1}A) \\
       {}+\nabla\Rm*\nabla^{2l-2} A+\Rm*\nabla^{2l-1}A
      \end{aligned} \\
   & =\sum_{i=0}^{2l-3}\nabla^{2l-1-i}\Rm*\nabla^i A+\nabla\Rm*\nabla^{2l-2} A+\nabla\Rm*\nabla^{2l-2}A+\Rm*\nabla^{2l-1}A \\
   & =\sum_{i=0}^{2l-1}\nabla^{2l-1-i}*\nabla^i A.
 \end{align*}
 Next,
 \begin{align*}
  \nabla\Delta^l A
  & =\nabla\Delta\Delta^{l-1}A \\
  & =\Delta\nabla\Delta^{l-1}A+\nabla\Rm*\nabla^{2l-2}A+\Rm*\nabla^{2l-1}A \\
  & =\Delta\left(\Delta^{l-1}\nabla A+\sum_{i=0}^{2l-1}\nabla^{2l-1-i}\Rm*\nabla^i A\right)+\nabla\Rm*\nabla^{2l-2}A+\Rm*\nabla^{2l-1}A \\
  & =\Delta^l\nabla A+\sum_{i=0}^{2l-1}\nabla^{2l-1-i}*\nabla^i A+\nabla\Rm*\nabla^{2l-2}A+\Rm*\nabla^{2l-1}A \\
  & =\Delta^l\nabla A+\sum_{i=0}^{2l-1}\nabla^{2l-1-i}*\nabla^i A.
 \end{align*}
 Assume the proposition holds for $k-1$. Then
 \begin{align*}
  \nabla\sum_{i=0}^{2l+k-3}\nabla^{2l+k-3-i}\Rm*\nabla^i A
  & =\sum_{i=0}^{2l+k-3}\nabla^{2l+k-2-i}\Rm*\nabla^i A+\Rm*\nabla^{2l+k-2}A \\
  & =\sum_{i=0}^{2l+k-2}\nabla^{2l+k-2-i}\Rm*\nabla^i A.
 \end{align*}
 Lastly,
 \begin{align*}
  \nabla^k\Delta^l A
  & =\nabla\nabla^{k-1}\Delta^l A \\
  & =\nabla\left(\Delta^l\nabla^{k-1}A+\sum_{i=0}^{2l+k-3}\nabla^{2l+k-3-i}\Rm*\nabla^i A\right) \\
  & =\Delta^l\nabla\nabla^{k-1}A+\sum_{i=0}^{2l-1}\nabla^{2l-1-i}\Rm*\nabla^i\nabla^{k-1}A+\nabla\sum_{i=0}^{2l+k-3}\nabla^{2l+k-3-i}\Rm*\nabla^i A \\
  & =\Delta^l\nabla^k A+\sum_{i=0}^{2l-1}\nabla^{2l-1-i}\Rm*\nabla^{i+k-1}A+\sum_{i=0}^{2l+k-2}\nabla^{2l+k-2-i}\Rm*\nabla^i A \\
  & =\Delta^l\nabla^k A+\sum_{i=0}^{2l+k-2}\nabla^{2l+k-2-i}\Rm*\nabla^i A.
 \end{align*}
\end{proof}
\begin{proposition} \label{prop: commutator of partial_t and nabla^k}
 Let $M$ be a manifold and $g(t)$ be a one-parameter family of metrics on $M$. If $A$ is a tensor on $M$ and $k\geq 1$, then
 \begin{equation*}
  \partial_t\nabla^k A=\nabla^k\partial_t A+\sum_{j=0}^{k-1}\nabla^j(\nabla\partial_t g*\nabla^{k-1-j}A).
 \end{equation*}
\end{proposition}
\begin{proof}
 First
 \begin{align*}
  \partial_t\nabla A
  & \begin{aligned}[t]
     =\partial_i\partial_t A_{j_1\cdots j_r}^{k_1\cdots k_s}
     -\sum_{m=1}^r\left[\partial_t\Gamma_{ij_m}^l A_{j_1\cdots j_{m-1}lj_{m+1}\cdots j_r}^{k_1\cdots k_s}
     +\Gamma_{ij_m}^l\partial_t A_{j_1\cdots j_{m-1}lj_{m+1}\cdots j_r}^{k_1\cdots k_s}\right] \\
     {}+\sum_{p=1}^s\left[\partial_t\Gamma_{il}^{k_p}A_{j_1\cdots j_r}^{k_1\cdots k_{p-1}qk_{p+1}\cdots k_s}
     +\Gamma_{il}^{k_p}\partial_t A_{j_1\cdots j_r}^{k_1\cdots k_{p-1}qk_{p+1}\cdots k_s}\right]
    \end{aligned} \\
  & =\nabla\partial_t A+\partial_t\Gamma*A \\
  & =\nabla\partial_t A+\nabla\partial_t g*A,
 \end{align*}
 so the proposition is true when $k=1$. Assume the proposition holds for $k-1$. Then
 \begin{align*}
  \partial_t\nabla^k A
  & =\partial_t\nabla\nabla^{k-1}A \\
  & =\nabla\partial_t\nabla^{k-1} A+\nabla\partial_t g*\nabla^{k-1}A \\
  & =\nabla^k\partial_t A+\nabla\partial_t g*\nabla^{k-1}A+\nabla\sum_{j=0}^{k-2}\nabla^j(\nabla\partial_t g*\nabla^{k-2-j}A) \\
  & =\nabla^k\partial_t A+\nabla\partial_t g*\nabla^{k-1}A+\nabla\sum_{j=0}^{k-2}\sum_{i=0}^j\nabla^{i+1}\partial_t g*\nabla^{k-2-i}A \\
  & =\nabla^k\partial_t A+\nabla\partial_t g*\nabla^{k-1}A+\sum_{j=0}^{k-2}\sum_{i=0}^j(\nabla^{i+2}\partial_t g*\nabla^{k-2-i}A+\nabla^{i+1}\partial_t g*\nabla^{k-1-i}A) \\
  & =\nabla^k\partial_t A+\sum_{j=0}^{k-2}\sum_{i=0}^j\nabla^{i+1}\partial_t g*\nabla^{k-1-i}A+\nabla\partial_t g*\nabla^{k-1}A+\sum_{j=0}^{k-2}\nabla^{j+2}\partial_t g*\nabla^{k-2-j}A \\
  & =\nabla^k\partial_t A+\sum_{j=0}^{k-2}\sum_{i=0}^j\nabla^{i+1}\partial_t g*\nabla^{k-1-i}A+\nabla\partial_t g*\nabla^{k-1}A+\sum_{j=1}^{k-1}\nabla^{j+1}\partial_t g*\nabla^{k-1-j}A \\
  & =\nabla^k\partial_t A+\sum_{j=0}^{k-2}\sum_{i=0}^j\nabla^{i+1}\partial_t g*\nabla^{k-1-i}A+\sum_{j=0}^{k-1}\nabla^{j+1}\partial_t g*\nabla^{k-1-j}A \\
  & =\nabla^k\partial_t A+\sum_{j=0}^{k-2}\nabla^j(\nabla\partial_t g*\nabla^{k-1-j}A)+\nabla^{k-1}(\nabla\partial_t g*A) \\
  & =\nabla^k\partial_t A+\sum_{j=0}^{k-1}\nabla^j(\nabla\partial_t g*\nabla^{k-1-j}A).
 \end{align*}
\end{proof}

\subsection{Evolution Equations}
We derive the equations for $\partial_t\nabla^k\Rm$ for every $k\geq 0$.
\begin{proposition} \label{prop: evolution of Rm}
 If $(M,g(t))$ is a solution to AOF, then
 \begin{equation*}
  \partial_t\Rm=\frac{(-1)^{\frac{n}{2}+1}}{2(n-2)}\Delta^\frac{n}{2}\Rm+\sum_{j=2}^{n/2+1}P_j^{n-2j+2}(\Rm).
 \end{equation*}
\end{proposition}
\begin{proof}
 Let $\hat{g}(t)$ be a one-parameter family of metrics on $M$ and $h=\partial_t\hat{g}$. The evolution of $\Rm$ is given by (\cite{Hamilton3MfdPositiveRicciCurvature}, Theorem 7.1)
 \begin{equation*}
  \partial_t R_{ijkl}=\tfrac{1}{2}[\nabla_i\nabla_kh_{jl}+\nabla_j\nabla_lh_{ik}-\nabla_i\nabla_lh_{jk}-\nabla_j\nabla_kh_{il}]+\Rm*h.
 \end{equation*}
 If $h=\Delta^{\frac{n}{2}-1}\Rc$ then, using Proposition \ref{prop: grad and Laplacian commutator} in the second line and Lemma \ref{lem: Delta Rc Bianchi identity} in the third line,
 \begin{align*}
  \partial_t R_{ijkl}
  & =\tfrac{1}{2}[\nabla_i\nabla_k\Delta^{\frac{n}{2}-1}R_{jl}+\nabla_j\nabla_l\Delta^{\frac{n}{2}-1}R_{ik}-\nabla_i\nabla_l\Delta^{\frac{n}{2}-1}R_{jk}-\nabla_j\nabla_k\Delta^{\frac{n}{2}-1}R_{il}]+\Rm*\Delta^{\frac{n}{2}-1}\Rc \\
  & =\tfrac{1}{2}\Delta^{\frac{n}{2}-1}
  [\nabla_i\nabla_kR_{jl}+\nabla_j\nabla_lR_{ik}-\nabla_i\nabla_lR_{jk}-\nabla_j\nabla_kR_{il}]+\sum_{i=0}^{n-2}\nabla^{n-2-i}\Rm*\nabla^i\Rc+P_2^{n-2}(\Rm) \\
  & =\tfrac{1}{2}\Delta^{\frac{n}{2}-1}[-\Delta R_{ijkl}+\Rm^{*2}]+P_2^{n-2}(\Rm) \\
  & =-\tfrac{1}{2}\Delta^\frac{n}{2} R_{ijkl}+P_2^{n-2}(\Rm).
 \end{align*}
 If $h=\Delta^{\frac{n}{2}-2}\nabla^2 R$ then, using Proposition \ref{prop: grad and Laplacian commutator} in the second and fourth lines,
 \begin{align*}
  \partial_tR_{ijkl}
  & =\begin{aligned}[t]
      \tfrac{1}{2}[\nabla_i\nabla_k\Delta^{\frac{n}{2}-2}\nabla_j\nabla_lR
      +\nabla_j\nabla_l\Delta^{\frac{n}{2}-2}\nabla_i\nabla_kR
      -\nabla_i\nabla_l\Delta^{\frac{n}{2}-2}\nabla_j\nabla_kR
      -\nabla_j\nabla_k\Delta^{\frac{n}{2}-2}\nabla_i\nabla_lR] \\
      {}+\Rm*\Delta^{\frac{n}{2}-2}\nabla^2 R
     \end{aligned} \\
  & =\begin{aligned}[t]
      \tfrac{1}{2}\Delta^{\frac{n}{2}-2}[\nabla_i\nabla_k\nabla_j\nabla_lR
      +\nabla_j\nabla_l\nabla_i\nabla_kR
      -\nabla_i\nabla_l\nabla_j\nabla_kR
      -\nabla_j\nabla_k\nabla_i\nabla_lR] \\
      {}+\sum_{i=0}^{n-2}\nabla^{n-2-i}\Rm*\nabla^i\nabla^2R+P_2^{n-2}(\Rm)
     \end{aligned} \\
   & =\begin{aligned}[t]
      \tfrac{1}{2}\Delta^{\frac{n}{2}-2}[\nabla_i\nabla_k\nabla_j\nabla_lR
      +\nabla_j\nabla_l\nabla_i\nabla_kR
      -\nabla_i\nabla_l\nabla_j\nabla_kR
      -\nabla_j\nabla_k\nabla_i\nabla_lR]
      +P_2^{n-2}(\Rm)
     \end{aligned} \\
   & =\begin{aligned}[t]
      \tfrac{1}{2}\Delta^{\frac{n}{2}-2}[\nabla_i\nabla_k\nabla_j\nabla_lR
      +\nabla_j\nabla_l\nabla_i\nabla_kR
      -\nabla_i\nabla_k\nabla_j\nabla_lR
      -\nabla_j\nabla_l\nabla_i\nabla_kR
      +\nabla\Rm*\nabla R+\Rm*\nabla^2R] \\
      {}+P_2^{n-2}(\Rm)
     \end{aligned} \\
    & =P_2^{n-2}(\Rm).
 \end{align*}
 If $h=\sum_{j=2}^{n/2}P_j^{n-2j}(\Rm)$, then
 \begin{align*}
  \partial_t\Rm
  & =\nabla^2\sum_{j=2}^{n/2}P_j^{n-2j}(\Rm)+\Rm*\sum_{j=2}^{n/2}P_j^{n-2j}(\Rm) \\
  & =\sum_{j=2}^{n/2}P_j^{n-2j+2}(\Rm)+\sum_{j=2}^{n/2}P_{j+1}^{n-2j}(\Rm).
 \end{align*}
 Combining these results, we conclude that if $h=\Obhat$ then
 \begin{align*}
  \partial_t\Rm
  & =\frac{(-1)^{\frac{n}{2}+1}}{2(n-2)}\Delta^\frac{n}{2}\Rm+P_2^{n-2}(\Rm)+P_2^{n-2}(\Rm)+\sum_{j=2}^{n/2}P_j^{n-2j+2}(\Rm)+\sum_{j=2}^{n/2}P_{j+1}^{n-2j}(\Rm) \\
  & =\frac{(-1)^{\frac{n}{2}+1}}{2(n-2)}\Delta^\frac{n}{2}\Rm+\sum_{j=2}^{n/2+1}P_j^{n-2j+2}(\Rm).
 \end{align*}
\end{proof}

\begin{proposition} \label{prop: evolution of nabla^k Rm}
 If $(M,g(t))$ is a solution to AOF, then
 \begin{equation*}
  \partial_t\nabla^k\Rm=\frac{(-1)^{\frac{n}{2}+1}}{2(n-2)}\Delta^\frac{n}{2}\nabla^k\Rm+\sum_{l=2}^{n/2+1}P_l^{n-2l+k+2}(\Rm).
 \end{equation*}
\end{proposition}
\begin{proof}
 We compute:
 \begin{align*}
  \sum_{j=0}^{k-1}\nabla^j(\nabla\partial_t g*\nabla^{k-1-j}\Rm)
  & =\sum_{j=0}^{k-1}\nabla^j\left(\sum_{l=1}^{n/2}P_l^{n-2l+1}(\Rm)*\nabla^{k-1-j}\Rm\right) \\
  & =\sum_{j=0}^{k-1}\nabla^j\sum_{l=1}^{n/2}P_{l+1}^{n-2l+k-j}(\Rm) \\
  & =\sum_{j=0}^{k-1}\sum_{l=1}^{n/2}P_{l+1}^{n-2l+k}(\Rm) \\
  & =\sum_{l=1}^{n/2}P_{l+1}^{n-2l+k}(\Rm) \\
  & =\sum_{l=2}^{n/2+1}P_l^{n-2l+k+2}(\Rm).
 \end{align*}
 Then, using Proposition \ref{prop: commutator of partial_t and nabla^k} in the first line, Proposition \ref{prop: evolution of Rm} in the second line, and Proposition \ref{prop: grad and Laplacian commutator} in the third line, we get
 \begin{align*}
  \partial_t\nabla^k\Rm
  & =\nabla^k\partial_t\Rm+\sum_{j=0}^{k-1}\nabla^j(\nabla\partial_t g*\nabla^{k-1-j}\Rm) \\
  & =\frac{(-1)^{\frac{n}{2}+1}}{2(n-2)}\nabla^k\Delta^\frac{n}{2}R_{ijkl}+\nabla^k\sum_{j=2}^{n/2+1}P_j^{n-2j+2}(\Rm)+\sum_{l=2}^{n/2+1}P_l^{n-2l+k+2}(\Rm) \\
  & =\frac{(-1)^{\frac{n}{2}+1}}{2(n-2)}\Delta^\frac{n}{2}\nabla^k R_{ijkl}+P_2^{n+k-2}(\Rm)+\sum_{j=2}^{n/2+1}P_j^{n-2j+k+2}(\Rm)+\sum_{l=2}^{n/2+1}P_l^{n-2l+k+2}(\Rm) \\
  & =\frac{(-1)^{\frac{n}{2}+1}}{2(n-2)}\Delta^\frac{n}{2}\nabla^k R_{ijkl}+\sum_{l=2}^{n/2+1}P_l^{n-2l+k+2}(\Rm).
\end{align*}
\end{proof}

\subsection{Short Time Existence and Uniqueness}
 We recall the short time existence and uniqueness theorems for the AOF. E. Bahuaud and D. Helliwell have shown in their papers \cite{BahuaudHelliwellExistence}, \cite{BahuaudHelliwellUniqueness} the following result:
\begin{theorem} \label{thm: short time existence}
 Let $h$ be a smooth metric on a compact manifold $M$ of even dimension $n\geq 4$. Then there is a unique smooth short time solution to the following flow:
 \begin{equation}
 \begin{cases}
  \partial_t g=\Obhat=(-1)^\frac{n}{2}\Ob+\frac{(-1)^\frac{n}{2}}{2(n-1)(n-2)}(\Delta^{\frac{n}{2}-1}R)g \label{eq: short time existence system} \\
  g(0)=h,
 \end{cases}
 \end{equation}
 where $\mathcal{O}$ is the ambient obstruction tensor on $M$ and $R$ is the scalar curvature of $M$.
\end{theorem}
\begin{proof}
We outline a proof of the existence theorem. Due to the diffeomorphism invariance of $M$, the system \eqref{eq: short time existence system} is not strongly parabolic. However, by choosing a vector field $W$ given by
\begin{equation*}
W=\frac{(-1)^{\frac{n}{2}-1}}{2(n-2)}\Delta^{\frac{n}{2}-1}X+\frac{(-1)^\frac{n}{2}}{4(n-1)}(\nabla\Delta^{\frac{n}{2}-2}R)^\sharp,
\end{equation*} 
where $X^k=g^{ij}(\Gamma_{ij}^k-\widetilde{\Gamma}_{ij}^k)$ and $\widetilde{\Gamma}$ is the connection of $h$, we obtain a strongly parabolic system:
\begin{equation}
\begin{cases}
\partial_t g=\Obhat+\mathcal{L}_Wg \label{eq: short time existence system adjusted} \\
g(0)=h.
\end{cases}
\end{equation}
We show this by computing the principal symbol $\sigma$ of the system \eqref{eq: short time existence system adjusted}. Let $A=-2\Rc+\mathcal{L}_X g$. We know from Proposition \ref{prop: Obhat expanded} that
\begin{equation}
\Obhat=\frac{(-1)^\frac{n}{2}}{n-2}\Delta^{\frac{n}{2}-1}\Rc+\frac{(-1)^{\frac{n}{2}-1}}{2(n-1)}\Delta^{\frac{n}{2}-2}\nabla^2 R+\sum_{j=2}^{n/2}P_j^{n-2j}(\Rm).  \label{eq: adjusted Ob expanded}
\end{equation}
We can linearize the first term of \eqref{eq: adjusted Ob expanded} as follows:
\begin{align*}
\partial_t\left[\frac{(-1)^\frac{n}{2}}{n-2}\Delta^{\frac{n}{2}-1}\Rc+\frac{(-1)^{\frac{n}{2}-1}}{2(n-2)}\Delta^{\frac{n}{2}-1}\mathcal{L}_X g\right]
& =\frac{(-1)^{\frac{n}{2}-1}}{2(n-2)}\partial_t\Delta^{\frac{n}{2}-1}A \\
& =\frac{(-1)^{\frac{n}{2}-1}}{2(n-2)}[g^{*(1-\frac{n}{2})}\partial_t\nabla^{\frac{n}{2}-1}A+\mathrm{lots}] \\
& =\frac{(-1)^{\frac{n}{2}-1}}{2(n-2)}\Delta^{\frac{n}{2}-1}\partial_t A+\mathrm{lots} \\
& =\frac{(-1)^{\frac{n}{2}-1}}{2(n-2)}\Delta^{\frac{n}{2}-1}\partial_t g+\mathrm{lots}.
\end{align*}
We used Proposition \ref{prop: commutator of partial_t and nabla^k} in the third line and the fact from Ricci flow (Chow-Knopf \cite{ChowKnopfRicciFlowIntro}, Theorem 3.13) that $\sigma[DA](\zeta)=|\zeta|^2$ in the fourth line. Let $Y=\frac{(-1)^\frac{n}{2}}{4(n-1)}(\nabla\Delta^{\frac{n}{2}-2}R)^\sharp$. The second term of \eqref{eq: adjusted Ob expanded} can be asorbed into $\mathcal{L}_W g$:
\begin{align*}
\frac{(-1)^{\frac{n}{2}-1}}{2(n-1)}\Delta^{\frac{n}{2}-2}\nabla_i\nabla_j R+(\mathcal{L}_Y g)_{ij}
& =\begin{aligned}[t]
\frac{(-1)^{\frac{n}{2}-1}}{4(n-1)}\Delta^{\frac{n}{2}-2}\nabla_i\nabla_j R+\frac{(-1)^{\frac{n}{2}-1}}{4(n-1)}\Delta^{\frac{n}{2}-2}\nabla_j\nabla_i R+\frac{(-1)^{\frac{n}{2}}}{4(n-1)}\Delta^{\frac{n}{2}-2}\nabla_i\nabla_j R \\
{}+\frac{(-1)^{\frac{n}{2}}}{4(n-1)}\Delta^{\frac{n}{2}-2}\nabla_j\nabla_i R+\mathrm{lots}
\end{aligned} \\
& =\mathrm{lots}.
\end{align*}
We commuted $\nabla_i$ and $\nabla_j$ and used Proposition \ref{prop: grad and Laplacian commutator} to commute $\Delta^{\frac{n}{2}-2}$ and $\nabla_i\nabla_j$.
So the principal symbol of the system \eqref{eq: short time existence system adjusted} is
\begin{align*}
\sigma[D(\Obhat+\mathcal{L}_Wg)](\zeta)=\frac{(-1)^{\frac{n}{2}-1}}{2(n-2)}|\zeta|^\frac{n}{2}.
\end{align*}
Since the PDE has order $n$ with respect to $g$, \eqref{eq: short time existence system adjusted} is strongly parabolic.

 So there exists $\epsilon>0$ for which the solution to \eqref{eq: short time existence system adjusted} exists for $t\in[0,\epsilon)$ via parabolic PDE theory.  Next, there exists a family $\phi_t:M\to M$ of diffeomorphisms satisfying
\begin{equation*}
\begin{cases}
\frac{\partial\phi_t}{\partial t} = -W(\phi_t,t) \\
\phi_0 = \textrm{id}_M.
\end{cases}
\end{equation*}
for $t\in[0,\epsilon)$. The existence of the $\phi_t$ follows from the existence and uniqueness theorem for nonautonomous ODE on manifolds, and the uniform $\epsilon$ follows from bounds on $W$ that result from the compactness of $M$. We now show that $\partial_t(\phi_t^*g)=\phi_t^*\partial_t g$. First, if $p\in M$ and $v_1,v_2\in T_pM$,
\begin{align*}
\lim_{s\to 0}\frac{\phi_{s+t}^*g(s+t)(v_1,v_2)-\phi_{s+t}^*g(t)(v_1,v_2)}{s} & =\lim_{s\to 0}\frac{g(s+t)_{ij}-g(t)_{ij}}{s}\lim_{s\to 0}[((\phi_{s+t})_*v_1)^i((\phi_{s+t})_*v_2)^j] \\
& =(\partial_t g)_{ij}(\phi_{t*}v_1)^i(\phi_{t*}v_2)^j \\
& =\phi_t^*\partial_t g(v_1,v_2).
\end{align*}
So
\begin{align*}
\partial_t(\phi_t^*g) & =\lim_{s\to 0}\frac{\phi_{s+t}^*g(s+t)-\phi_t^*g(t)}{s} \\
& =\lim_{s\to 0}\frac{\phi_{s+t}^*g(s+t)-\phi_{s+t}^*g(t)}{s}+\lim_{s\to 0}\frac{\phi_{s+t}^*g(t)-\phi_t^*g(t)}{s} \\
& =\phi_t^*\partial_t g+\partial_s|_{s=0}(\phi_{t+s}^*g(t)) \\
& =\phi_t^*[(-1)^\frac{n}{2}\Ob(g)+\tfrac{(-1)^\frac{n}{2}}{2(n-1)(n-2)}[\Delta^{\frac{n}{2}-1}R(g(t))]g(t)+\mathcal{L}_W g(t)]+\partial_s|_{s=0}[(\phi_t^{-1}\circ\phi_{t+s})^*\phi_t^*g(t)] \\
& =(-1)^\frac{n}{2}\Ob(\phi_t^*g(t))+\tfrac{(-1)^\frac{n}{2}}{2(n-1)(n-2)}[\Delta^{\frac{n}{2}-1}R(\phi_t^*g(t))]\phi_t^*g(t)+\phi_t^*(\mathcal{L}_W g(t))-\mathcal{L}_{[(\phi_t^{-1})_*W(t)]}(\phi_t^*g(t)) \\
& =(-1)^\frac{n}{2}\Ob(\phi_t^*g(t))+\tfrac{(-1)^\frac{n}{2}}{2(n-1)(n-2)}[\Delta^{\frac{n}{2}-1}R(\phi_t^*g(t))]\phi_t^*g(t).
\end{align*}
Since $\phi_0^*g(0)=g(0)=h$, $\phi_t^*g(t)$ satisfies \eqref{eq: short time existence system}. Therefore these diffeomorphisms pull back the short time solution of \eqref{eq: short time existence system adjusted} to give a solution of \eqref{eq: short time existence system} that exists for  $t\in[0,\epsilon)$.
\end{proof}

\section{Local Integral Estimates} \label{sec: Local Integral Estimates}
In this section, let $(M^n,g)$ be a Riemannian manifold that is a solution to the AOF on a time interval $[0,T)$. We give local $L^2$ estimates for $\nabla^k\Rm$ for all $k\in\N$. We need to use local $L^2$ estimates since we can only convert $L^2$ estimates to pointwise estimates locally. These local pointwise estimates are used in the proof of the pointwise smoothing estimates given in Theorem \ref{thm:C^m smoothing estimate along AOF}. Specify the Laplace operator by $\Delta=-\nabla^*\nabla$. Let $\phi\in C_c^\infty(M)$ be a cutoff function with constants $\Lambda,\Lambda_1>0$ such that
\begin{equation*}
 \sup_{t\in[0,T)}|\nabla\phi|\leq\Lambda_1,\quad \max_{0\leq i\leq\frac{n}{2}}\sup_{t\in[0,T)}|\nabla^i\phi|\leq \Lambda.
\end{equation*}
\begin{lemma} \label{iterated Laplacian inner product}
 Suppose $M,\phi$ satisfy the above hypotheses. Let $A$ be any tensor and $p\geq 1,q\geq 2$. Then
 \begin{equation*}
  \int_M\phi^p\langle\Delta^qA,A\rangle=(-1)^q\int_{[\phi>0]}\sum_{i=0}^q P_p^{q-i}(\phi)*\nabla^i A*\nabla^q A+\int_M\sum_{i=0}^{2q-2}\phi^p\nabla^{2q-2-i}\Rm*\nabla^i A*A.
 \end{equation*}
\end{lemma}
\begin{proof}
 We first claim that if $q\geq 2$, then
 \begin{equation*}
  \Delta^q A=(-1)^q(\nabla^*)^q\nabla^q A+\sum_{i=0}^{2q-2}\nabla^{2q-2-i}\Rm*\nabla^i A.
 \end{equation*}
 If $q=2$, we get, using Proposition \ref{prop: grad and Laplacian commutator}, that
 \begin{align*}
  \Delta^2 A & =-\nabla^*\nabla\Delta A \\
  & = -\nabla^*\Delta\nabla A+\nabla^*[\nabla\Rm*A+\Rm*\nabla A] \\
  & =(\nabla^*)^2\nabla^2A+\nabla^2\Rm*A+\nabla\Rm*\nabla A+\Rm*\nabla^2 A,
 \end{align*}
 which agrees with the claim. Suppose the claim is true for every integer less than $q$. First,
 \begin{align*}
  \Delta^q A &= -\nabla^*\nabla\Delta^{q-1}A \\
  & =-\nabla^*\left[\Delta^{q-1}\nabla A+\sum_{i=0}^{2q-3}\nabla^{2q-3-i}\Rm*\nabla^i A\right] \\
  & =-\nabla^*\Delta^{q-1}\nabla A+\sum_{i=0}^{2q-3}\left[\nabla^{2q-2-i}\Rm*\nabla^i A+\nabla^{2q-3-i}\Rm*\nabla^{i+1}A\right] \\
  & =-\nabla^*\Delta^{q-1}\nabla A+\sum_{i=0}^{2q-3}\nabla^{2q-2-i}\Rm*\nabla^i A+\sum_{i=1}^{2q-2}\nabla^{2q-2-i}\Rm*\nabla^i A \\
  & =-\nabla^*\Delta^{q-1}\nabla A+\sum_{i=0}^{2q-2}\nabla^{2q-2-i}\Rm*\nabla^i A.
 \end{align*}
 Applying the last equation above and then the inductive hypothesis, 
 \begin{align*}
  \Delta^q A & =-\nabla^*\Delta^{q-1}\nabla A+\sum_{i=0}^{2q-2}\nabla^{2q-2-i}\Rm*\nabla^i A \\
  & =-\nabla^*\left[(-1)^{q-1}(\nabla^*)^{q-1}\nabla^{q-1}\nabla A+\sum_{i=0}^{2q-4}\nabla^{2q-4-i}\Rm*\nabla^i\nabla A\right]+\sum_{i=0}^{2q-2}\nabla^{2q-2-i}\Rm*\nabla^i A \\
  & =(-1)^q(\nabla^*)^q\nabla^q A+\sum_{i=0}^{2q-4}\nabla^{2q-3-i}\Rm*\nabla^{i+1} A+\sum_{i=0}^{2q-4}\nabla^{2q-4-i}\Rm*\nabla^{i+2} A+\sum_{i=0}^{2q-2}\nabla^{2q-2-i}\Rm*\nabla^i A \\
  & =(-1)^q(\nabla^*)^q\nabla^q A+\sum_{i=1}^{2q-3}\nabla^{2q-2-i}\Rm*\nabla^i A+\sum_{i=2}^{2q-2}\nabla^{2q-2-i}\Rm*\nabla^i A+\sum_{i=0}^{2q-2}\nabla^{2q-2-i}\Rm*\nabla^i A \\
  & =(-1)^q(\nabla^*)^q\nabla^q A+\sum_{i=0}^{2q-2}\nabla^{2q-2-i}\Rm*\nabla^i A.
 \end{align*}
 This proves the claim. We compute
 \begin{align*}
  (-1)^{q+1}\int_M\nabla^q A*\nabla^q(\phi^p A)
  & =(-1)^{q+1}\int_M\nabla^q A*\sum_{i=0}^q\nabla^{q-i}(\phi^p)*\nabla^i A \\
  & =(-1)^{q+1}\int_{[\phi>0]}\sum_{i=0}^q\sum_{|\alpha|=q-i}\nabla^{\alpha_1}\phi_1*\cdots*\nabla^{\alpha_p}\phi_p*\nabla^i A*\nabla^q A \\
  & =(-1)^{q+1}\int_{[\phi>0]}\sum_{i=0}^q P_p^{q-i}(\phi)*\nabla^i A*\nabla^q A.
 \end{align*}
 Finally, applying the claim,
 \begin{align*}
  \int_M\phi^p\langle\Delta^qA,A\rangle
  & =\int_M\phi^p\left\langle(-1)^q(\nabla^*)^q\nabla^q A+\sum_{i=0}^{2q-2}\nabla^{2q-2-i}\Rm*\nabla^i A,A\right\rangle \\
  & =(-1)^q\int_M\nabla^q A*\nabla^q(\phi^p A)+\int_M\sum_{i=0}^{2q-2}\phi^p\nabla^{2q-2-i}\Rm*\nabla^i A*A \\
  & =(-1)^q\int_{[\phi>0]}\sum_{i=0}^q P_p^{q-i}(\phi)*\nabla^i A*\nabla^q A+\int_M\sum_{i=0}^{2q-2}\phi^p\nabla^{2q-2-i}\Rm*\nabla^i A*A.
 \end{align*}
\end{proof}

\begin{proposition} \label{prop: time derivative for local L2 norm of nabla^k Rm}
 Suppose $M,\phi$ satisfy the above hypotheses. If $p\geq 1,k\geq 0$, then
 \begin{multline} \label{eq: time derivative eqn for local L2 norm of nabla^k Rm}
  \frac{\partial}{\partial t}\int_M\phi^p|\nabla^k\Rm|^2
  = -\frac{1}{n-2}\int\limits_M\phi^p|\nabla^{\frac{n}{2}+k}\Rm|^2+\int\limits_M\phi^p\sum_{l=k}^{\frac{n}{2}+k-1}P_{\frac{n}{2}+k-l+2}^{2l}(\Rm) \\
  +\int\limits_{[\phi>0]}\sum_{i=0}^{\frac{n}{2}-1}P_p^{\frac{n}{2}-i}(\phi)*\nabla^{k+i}\Rm*\nabla^{k+\frac{n}{2}}\Rm.
 \end{multline}
\end{proposition}
\begin{proof}
 First, we have
 \begin{equation*}
  \frac{\partial}{\partial t}\int_M\phi^p|\nabla^k\Rm|^2\,dV_g=2\int_M\phi^p\left\langle\frac{\partial}{\partial t}\nabla^k\Rm,\nabla^k\Rm\right\rangle\,dV_g+\int_M\phi^p|\nabla^k\Rm|^2\frac{\partial g}{\partial t}\,dV_g.
 \end{equation*}
 We can expand the first integral by substituting Proposition \ref{prop: evolution of nabla^k Rm}, which states that for our flow,
 \begin{equation*}
  \frac{\partial}{\partial t}\nabla^k\Rm=\frac{(-1)^{\frac{n}{2}+1}}{2(n-2)}\Delta^\frac{n}{2}\nabla^k\Rm+\sum_{i=2}^{\frac{n}{2}+1}P_i^{n-2i+k+2}(\Rm).
 \end{equation*}
 Applying Lemma \ref{iterated Laplacian inner product} to the first term of $\frac{\partial}{\partial t}\nabla^k\Rm$ gives that
 \begin{align*}
  \frac{(-1)^{\frac{n}{2}+1}}{n-2}\int\limits_M\phi^p\langle\Delta^\frac{n}{2}\nabla^k\Rm,\nabla^k\Rm\rangle
  & =\begin{aligned}[t]
      & \frac{(-1)^{n+1}}{n-2}\int\limits_{[\phi>0]}\sum_{i=0}^\frac{n}{2}P_p^{\frac{n}{2}-i}(\phi)*\nabla^{k+i}\Rm*\nabla^{k+\frac{n}{2}}\Rm \\
      & {}+\int\limits_M\sum_{i=0}^{n-2}\phi^p\nabla^{n-2-i}\Rm*\nabla^{k+i}\Rm*\nabla^k\Rm
     \end{aligned} \\
  & =\begin{aligned}[t]
      & -\frac{1}{n-2}\int\limits_M\phi^p|\nabla^{\frac{n}{2}+k}\Rm|^2 \\
      & {}+\int\limits_{[\phi>0]}\sum_{i=0}^{\frac{n}{2}-1}P_p^{\frac{n}{2}-i}(\phi)*\nabla^{k+i}\Rm*\nabla^{k+\frac{n}{2}}\Rm \\
      & {}+\int\limits_M\phi^p P_3^{n+2k-2}(\Rm).
     \end{aligned}
 \end{align*}
 Substituting the second term of $\frac{\partial}{\partial t}\nabla^k\Rm$ into the inner product gives that
 \begin{align*}
  \int\limits_M\phi^p\left\langle\nabla^k\Rm,\sum_{i=2}^{\frac{n}{2}+1}P_i^{n-2i+k+2}(\Rm)\right\rangle
  & =\int\limits_M\phi^p\sum_{i=3}^{\frac{n}{2}+2}P_i^{n-2i+k+4}(\Rm) \\
  & =\int\limits_M\phi^p\sum_{l=k}^{\frac{n}{2}+k-1}P_{\frac{n}{2}+k-l+2}^{2l}(\Rm).
 \end{align*}
 Since 
 \begin{align*}
  \frac{\partial g}{\partial t} &= \Delta^{\frac{n}{2}-1}\Rc+\Delta^{\frac{n}{2}-2}\nabla^2 R+\sum_{i=2}^\frac{n}{2}P_i^{n-2i}(\Rm) \\
  &=\nabla^{n-2}\Rm+\nabla^{n-4+2}\Rm+\sum_{i=2}^\frac{n}{2}P_i^{n-2i}(\Rm) \\
  &=\sum_{i=1}^\frac{n}{2}P_i^{n-2i}(\Rm),
 \end{align*}
 we have
 \begin{align*}
  \int_M\phi^p|\nabla^k\Rm|^2\frac{\partial g}{\partial t}
  & =\int_M\phi^p(\nabla^k\Rm)^{*2}\sum_{i=1}^\frac{n}{2}P_i^{n-2i}(\Rm) \\
  & =\int_M\phi^p\sum_{i=1}^\frac{n}{2}P_{i+2}^{n-2i+2k}(\Rm) \\
  & =\int_M\phi^p\sum_{i=3}^{\frac{n}{2}+2}P_i^{n-2i+2k+4}(\Rm) \\
  & =\int_M\phi^p\sum_{l=k}^{\frac{n}{2}+k-1}P_{\frac{n}{2}+k-l+2}^{2l}(\Rm).
 \end{align*}
 Combining all of these results yields
\begin{align*}
\begin{split}
\frac{\partial}{\partial t}\int\limits_M\phi^p|\nabla^k\Rm|^2 ={}
& -\frac{1}{n-2}\int\limits_M\phi^p|\nabla^{\frac{n}{2}+k}\Rm|^2+\int\limits_{[\phi>0]}\sum_{i=0}^{\frac{n}{2}-1}P_p^{\frac{n}{2}-i}(\phi)*\nabla^{k+i}\Rm*\nabla^{k+\frac{n}{2}}\Rm \\
& {}+\int\limits_M\phi^p P_3^{n+2k-2}(\Rm)+\int\limits_M\phi^p\sum_{l=k}^{\frac{n}{2}+k-1}P_{\frac{n}{2}+k-l+2}^{2l}(\Rm)+\int\limits_M\phi^p\sum_{l=k}^{\frac{n}{2}+k-1}P_{\frac{n}{2}+k-l+2}^{2l}(\Rm)
\end{split} \\
\begin{split}
={} & -\frac{1}{n-2}\int\limits_M\phi^p|\nabla^{\frac{n}{2}+k}\Rm|^2+\int\limits_{[\phi>0]}\sum_{i=0}^{\frac{n}{2}-1}P_p^{\frac{n}{2}-i}(\phi)*\nabla^{k+i}\Rm*\nabla^{k+\frac{n}{2}}\Rm \\
& {}+\int\limits_M\phi^p\sum_{l=k}^{\frac{n}{2}+k-1}P_{\frac{n}{2}+k-l+2}^{2l}(\Rm).
\end{split}
 \end{align*}
\end{proof}

We estimate the last two terms of \eqref{eq: time derivative eqn for local L2 norm of nabla^k Rm}. First, we recall two corollaries from the paper \cite{KuwertSchaetzleGradientFlowWillmoreFunctional} of E. Kuwert and R. Sch\"{a}tzle.
\begin{proposition} \label{prop: K-S Corollary 5.2}
 (\cite{KuwertSchaetzleGradientFlowWillmoreFunctional}, Corollary 5.2) Suppose $M,\phi$ satisfy the above hypotheses. Let $A$ be a tensor. For $2\leq p<\infty$, $s\geq p$, and $c=c(n,p,s,\Lambda_1)$,
 \begin{equation*}
  \left(\int_M|\nabla A|^p\phi^s\right)^\frac{1}{p}
  \leq\epsilon\left(\int_M|\nabla^2 A|^p\phi^{s+p}\right)^\frac{1}{p}
  +\frac{c}{\epsilon}\left(\int_{[\phi>0]}|A|^p\phi^{s-p}\right)^\frac{1}{p}.
 \end{equation*}
\end{proposition}
\begin{proposition} \label{prop: K-S Corollary 5.5}
(\cite{KuwertSchaetzleGradientFlowWillmoreFunctional}, Corollary 5.5) Suppose $M,\phi$ satisfy the above hypotheses. Let $A$ be a tensor. Let $0\leq i_1,\dots,i_r\leq k$, $i_1+\dots+i_r=2k$, and $s\geq 2k$. Then we have
 \begin{equation*}
  \left|\int_M\phi^s\nabla^{i_1}A*\dots*\nabla^{i_r}A\right|\leq c\|A\|_\infty^{r-2}\left(\int_M\phi^s|\nabla^k A|^2\,dV+\|A\|_{2,[\phi>0]}^2\right),
 \end{equation*}
 where $c=c(k,n,r,s,\Lambda_1)$.
\end{proposition}

We estimate the last term of \eqref{eq: time derivative eqn for local L2 norm of nabla^k Rm}.
\begin{lemma} \label{lem: nabla(l) to nabla(l+1) interpolation inequality}
 Suppose $M,\phi$ satisfy the above hypotheses. If $l\geq 1,q\geq 0$, then for every $\epsilon>0$,
 \begin{equation*}
  \int_M\phi^{2l+q}|\nabla^l\Rm|^2\leq\epsilon\int_M\phi^{2l+q+2}|\nabla^{l+1}\Rm|^2+\frac{C}{\epsilon^l}\int_{[\phi>0]}\phi^q|\Rm|^2.
 \end{equation*}
 where $C=C(n,l,\Lambda_1,q)$.
\end{lemma}
\begin{proof}
 If $l=1$, the inequality follows immediately from Proposition \ref{prop: K-S Corollary 5.2}. Assume that the the inequality is true for all integers at most $l$. Then
 \begin{align*}
  \int\phi^{2l+2+q}|\nabla^{l+1}\Rm|^2
  & \leq\frac{\epsilon}{2}\int\phi^{2l+4+q}|\nabla^{l+2}\Rm|^2+\frac{C}{\epsilon}\int\phi^{2l+q}|\nabla^l\Rm|^2 \\
  & \leq\frac{\epsilon}{2}\int\phi^{2l+4+q}|\nabla^{l+2}\Rm|^2+\frac{C}{\epsilon}\frac{\epsilon}{2C}\int\phi^{2l+q+2}|\nabla^{l+1}\Rm|^2+\frac{C}{\epsilon}\frac{C}{\epsilon^l}\int\phi^q|\Rm|^2 \\
  & =\frac{\epsilon}{2}\int\phi^{2l+4+q}|\nabla^{l+2}\Rm|^2+\frac{1}{2}\int\phi^{2l+q+2}|\nabla^{l+1}\Rm|^2+\frac{C}{\epsilon^{l+1}}\int\phi^q|\Rm|^2.
 \end{align*}
 Collecting terms, we see that the statement is also true for $l+1$.
\end{proof}

\begin{lemma} \label{lem: nabla^l to nabla^q interpolation inequality}
 Suppose $M,\phi$ satisfy the above hypotheses. If $q\geq 0$ and $0\leq l\leq q$,
 \begin{equation*}
  \int_M\phi^{2l+r}|\nabla^l\Rm|^2\leq\epsilon^{q-l}\int_M\phi^{2q+r}|\nabla^q\Rm|^2+C\epsilon^{-l}\int_{[\phi>0]}\phi^r|\Rm|^2.
 \end{equation*}
 where $C=C(n,l,\Lambda_1,r)$.
\end{lemma}
\begin{proof}
 Let $m=q-l$. The desired inequality is equivalent to 
 \begin{equation} \label{eq: ineq inside pf of nabla^l to nabla^q interpolation inequality lemma}
  \int_M\phi^{2q-2m+r}|\nabla^{q-m}\Rm|^2\leq\epsilon^m\int_M\phi^{2q+r}|\nabla^q\Rm|^2+C\epsilon^{m-q}\int_{[\phi>0]}\phi^r|\Rm|^2.
 \end{equation}
 We prove this inequality by induction on $m$. If $m=0$ the inequality is true:
 \begin{equation*}
  \int_M\phi^{2q+r}|\nabla^q\Rm|^2\leq\int_M\phi^{2q+r}|\nabla^q\Rm|^2+C\epsilon^{-q}\int_{[\phi>0]}\phi^r|\Rm|^2.
 \end{equation*}
 Assume the inequality \eqref{eq: ineq inside pf of nabla^l to nabla^q interpolation inequality lemma} is true for every integer less than $m$. Then
 \begin{align*}
  \int_M\phi^{2q-2m+r}|\nabla^{q-m}\Rm|^2
  & \leq\epsilon\int_M\phi^{2q-2m+r+2}|\nabla^{q-m+1}\Rm|^2+C\epsilon^{m-q}\int_{[\phi>0]}\phi^r|\Rm|^2 \\
  & \leq\epsilon\epsilon^{m-1}\int_M\phi^{2q+r}|\nabla^q\Rm|^2+\epsilon C\epsilon^{m-q-1}\int_{[\phi>0]}\phi^r|\Rm|^2+C\epsilon^{m-q}\int_{[\phi>0]}\phi^r|\Rm|^2 \\
  & =\epsilon^m\int_M\phi^{2q+r}|\nabla^q\Rm|^2+C\epsilon^{m-q}\int_{[\phi>0]}\phi^r|\Rm|^2.
 \end{align*}
 We applied Lemma \ref{lem: nabla(l) to nabla(l+1) interpolation inequality} in the first line and the inductive hypothesis in the second line.
\end{proof}

\begin{lemma} \label{lem: k dependent estimate for term containing P(phi)}
 Suppose $M,\phi$ satisfy the above hypotheses. Let $0\leq i\leq\frac{n}{2}-1$ and $p\geq n+2k$. Then for every $\delta>0$,
 \begin{equation*}
  \int_M P_p^{\frac{n}{2}-i}(\phi)*\nabla^{i+k}\Rm*\nabla^{\frac{n}{2}+k}\Rm
  \leq C\delta\int_M\phi^p|\nabla^{\frac{n}{2}+k}\Rm|^2+C\delta^\frac{-n-2i-4k}{n-2i}\int_{[\phi>0]}\phi^{p-n-2k}|\Rm|^2,
 \end{equation*}
 where $C=C(n,k,p,\Lambda,i)$.
\end{lemma}
\begin{proof} We apply the Cauchy-Schwarz inequality:
 \begin{align*}
  \int_M P_p^{\frac{n}{2}-i}(\phi)*\nabla^{i+k}\Rm*\nabla^{\frac{n}{2}+k}\Rm
  & \leq C(\Lambda)\int_M|\phi^{p-(\frac{n}{2}-i)}*\nabla^{i+k}\Rm*\nabla^{\frac{n}{2}+k}\Rm| \\
  & \leq C\epsilon^\beta\int_M\phi^p|\nabla^{\frac{n}{2}+k}\Rm|^2+C\epsilon^{-\beta}\int_{[\phi>0]}\phi^{p-n+2i}|\nabla^{i+k}\Rm|^2.
 \end{align*}
 The second term can be estimated using Lemma \ref{lem: nabla^l to nabla^q interpolation inequality}:
 \begin{align*}
  \int_{[\phi>0]}\phi^{p-n+2i}|\nabla^{i+k}\Rm|^2
  & =\int_{[\phi>0]}\phi^{2(i+k)+(p-n-2k)}|\nabla^{i+k}\Rm|^2 \\
  & \leq\epsilon^{\frac{n}{2}-i}\int_M|\nabla^{\frac{n}{2}+k}\Rm|^2+C\epsilon^{-i-k}\int_{[\phi>0]}\phi^{p-n-2k}|\Rm|^2.
 \end{align*}
 If $\beta=\frac{n}{2}-i-\beta$, then $\beta=\frac{n-2i}{4}$. If we set $\delta=\epsilon^\frac{n-2i}{4}$, then $\epsilon=\delta^\frac{4}{n-2i}$ and
 \begin{equation*}
  \epsilon^{-\beta-i-k}=\delta^{\frac{4}{n-2i}\left(\frac{2i-n}{4}-i-k\right)}=\delta^\frac{-n-2i-4k}{n-2i}.
 \end{equation*}
 Therefore
 \begin{align*}
  \int_M P_p^{\frac{n}{2}-i}(\phi)*\nabla^{i+k}\Rm*\nabla^{\frac{n}{2}+k}\Rm
  & \begin{aligned}[t]
     \leq C\epsilon^\beta\int_M\phi^p|\nabla^{\frac{n}{2}+k}\Rm|^2+C\epsilon^{-\beta+\frac{n}{2}-i}\int_M|\nabla^{\frac{n}{2}+k}\Rm|^2 \\
     {}+C\epsilon^{-\beta-i-k}\int_{[\phi>0]}\phi^{p-n-2k}|\Rm|^2
    \end{aligned} \\
  & \leq C\delta\int_M\phi^p|\nabla^{\frac{n}{2}+k}\Rm|^2+C\delta^\frac{-n-2i-4k}{n-2i}\int_{[\phi>0]}\phi^{p-n-2k}|\Rm|^2.
 \end{align*}
\end{proof}

We estimate the penultimate term of \eqref{eq: time derivative eqn for local L2 norm of nabla^k Rm}.
\begin{lemma} \label{lem: k dependent estimate for term containing P(Rm)}
 Suppose $M,\phi$ satisfy the above hypotheses. Let $K=\max\{1,\|\Rm\|_\infty\}$. If $p\geq n+2k$ and $k\leq l\leq\frac{n}{2}+k-l$, then for every $\delta$ satisfying $0<\delta\leq 1$,
 \begin{equation*}
  \int_M\phi^p P_{\frac{n}{2}+k-l+2}^{2l}(\Rm)\leq C\delta\int_M\phi^{p+n+2k-2l}|\nabla^{\frac{n}{2}+k}\Rm|^2+CK^{\frac{n}{2}+k}\delta^\frac{2l}{2l-n-2k}\|\Rm\|_{2,[\phi>0]}^2,
 \end{equation*}
 where $C=C(n,k,p,\Lambda_1,l)$.
\end{lemma}
\begin{proof} Since $p\geq n+2k\geq n+2k-2=2(\frac{n}{2}+k-1)$, Proposition \ref{prop: K-S Corollary 5.5} implies
 \begin{equation*}
  \int_M\phi^p P_{\frac{n}{2}+k-l+2}^{2l}(\Rm)\leq C\|\Rm\|_\infty^{\frac{n}{2}+k-l}\left(\int_M\phi^p|\phi^l\Rm|^2+\|\Rm\|_{2,[\phi>0]}^2\right)
 \end{equation*}
 Let $\epsilon=K^{-1}\delta^\frac{2}{n+2k-2l}$. We have $p-2l\geq n+2k-(n+2k-1)=1$. Via Lemma \ref{lem: nabla^l to nabla^q interpolation inequality},
 \begin{align*}
  C\|\Rm\|_\infty^{\frac{n}{2}+k-l}\int_M\phi^p|\nabla^l\Rm|^2
  & \begin{aligned}[t]
     \leq CK^{\frac{n}{2}+k-l}\epsilon^{\frac{n}{2}+k-l}\int_M\phi^{n+2k+p-2l}|\nabla^{\frac{n}{2}+k}\Rm|^2 \\
     {}+CK^{\frac{n}{2}+k-l}\epsilon^{-l}\int_{[\phi>0]}\phi^{p-2l}|\Rm|^2
    \end{aligned} \\
  & =C\delta\int_M\phi^{n+2k+p-2l}|\nabla^{\frac{n}{2}+k}\Rm|^2+CK^{\frac{n}{2}+k}\delta^\frac{2l}{2l-n-2k}\int_{[\phi>0]}\phi^{p-2l}|\Rm|^2.
 \end{align*}
 Since $k\leq l\leq\frac{n}{2}+k-l$ and $0<\delta\leq 1$, we get $\delta^\frac{2l}{2l-n-2k}\geq\delta^{-\frac{2k}{n}}\geq 1$ and $K^{\frac{n}{2}+k-l}\leq K^\frac{n}{2}$. Therefore
 \begin{align*}
  \int_M\phi^p P_{\frac{n}{2}+k-l+2}^{2l}(\Rm)
  & \begin{aligned}[t]
     \leq C\delta\int_M\phi^{n+2k+p-2l}|\nabla^{\frac{n}{2}+k}\Rm|^2+CK^{\frac{n}{2}+k}\delta^\frac{2l}{2l-n-2k}\int_{[\phi>0]}\phi^{p-2l}|\Rm|^2 \\
     {}+K^{\frac{n}{2}+k-l}\|\Rm\|_{2,[\phi>0]}^2
    \end{aligned} \\
  & \leq C\delta\int_M\phi^{p+n+2k-2l}|\nabla^{\frac{n}{2}+k}\Rm|^2+CK^{\frac{n}{2}+k}\delta^\frac{2l}{2l-n-2k}\|\Rm\|_{2,[\phi>0]}^2.
 \end{align*}
\end{proof}

\begin{proposition} \label{prop: interpolation type estimate for patial_t nabla^k Rm}
 Suppose $M,\phi$ satisfy the above hypotheses. Let $K=\max\{1,\|\Rm\|_\infty\}$. If $p\geq n+2k$, then for every $\delta$ satisfying $0<\delta\leq 1$,
  \begin{equation*}
   \partial_t\|\phi^\frac{p}{2}\nabla^k\Rm\|_2^2
   \leq -\tfrac{1}{2(n-2)}\|\phi^\frac{p}{2}\nabla^{\frac{n}{2}+k}\Rm\|_2^2+CK^{\frac{n}{2}+k}\|\Rm\|_{2,[\phi>0]}^2
  \end{equation*}
 where $C=C(n,k,p,\Lambda)$.
\end{proposition}
\begin{proof}
 Applying the estimates from Lemmas \ref{lem: k dependent estimate for term containing P(Rm)} and \ref{lem: k dependent estimate for term containing P(phi)} to the equation \eqref{eq: time derivative eqn for local L2 norm of nabla^k Rm} in Proposition \ref{prop: time derivative for local L2 norm of nabla^k Rm}, we obtain
 \begin{multline*}
  \partial_t\|\phi^\frac{p}{2}\nabla^k\Rm\|_2^2
  \leq -\tfrac{1}{n-2}\|\phi^\frac{p}{2}\nabla^{\frac{n}{2}+k}\Rm\|_2^2 \\
  +\sum_{l=k}^{\frac{n}{2}+k-1}\left[C_1\delta\|\phi^{\frac{p}{2}+\frac{n}{2}+k-l}\nabla^{\frac{n}{2}+k}\Rm|_2^2+C_1K^{\frac{n}{2}+k}\delta^\frac{2l}{2l-n-2k}\|\Rm\|_{2,[\phi>0]}^2\right] \\
  {}+\sum_{i=0}^{\frac{n}{2}-1}\left[C_2\delta\|\phi^\frac{p}{2}\nabla^{\frac{n}{2}+k}\Rm\|_2^2+C_2\delta^\frac{-n-2i-4k}{n-2i}\|\phi^{\frac{p}{2}-\frac{n}{2}-k}\Rm\|_{2,[\phi>0]}^2\right],
 \end{multline*}
 where $C_1=C_1(n,k,p,\Lambda,l)$ and $C_2=C_2(n,k,p,\Lambda_1,i)$. From the inequalities
 \begin{equation*}
  1-n-2k\leq 1-\frac{2n+4k}{n-2i}\leq -\frac{n+4k}{n},\quad \frac{2-n-2k}{2}\leq 1+\frac{n+2k}{2l-n-2k}\leq -\frac{2k}{n}
 \end{equation*}
 we conclude
 \begin{equation*}
  \max\left(\{\delta^\frac{2l}{2l-n-2k}:k\leq l\leq\tfrac{n}{2}+k-1\}\cup\{\delta^\frac{-n-2i-4k}{n-2i}:0\leq i\leq\tfrac{n}{2}-1\}\right)=\delta^{1-n-2k}.
 \end{equation*}
 Therefore
 \begin{align*}
  \partial_t\|\phi^\frac{p}{2}\nabla^k\Rm\|_2^2
  & \leq -\tfrac{1}{n-2}\|\phi^\frac{p}{2}\nabla^{\frac{n}{2}+k}\Rm\|_2^2
     +\widetilde{C}\delta\|\phi^\frac{p}{2}\nabla^{\frac{n}{2}+k}\Rm\|_2^2+\widetilde{C}K^{\frac{n}{2}+k}\delta^{1-n-2k}\|\Rm\|_{2,[\phi>0]}^2 \\
  & \leq -\tfrac{1}{2(n-2)}\|\phi^\frac{p}{2}\nabla^{\frac{n}{2}+k}\Rm\|_2^2+CK^{\frac{n}{2}+k}\|\Rm\|_{2,[\phi>0]}^2,
 \end{align*}
 where
 \begin{equation*}
  \widetilde{C}\equiv\sum_{l=k}^{\frac{n}{2}+k-1}C_1+\sum_{i=0}^{\frac{n}{2}-1}C_2,\quad\delta\equiv\min\{\tfrac{1}{2(n-2)}\widetilde{C}^{-1},1\}.
 \end{equation*}
\end{proof}
\begin{proposition} \label{prop: integral BBS estimate for multiples of n/2}
 Suppose $M,\phi$ satisfy the above hypotheses. Suppose $\max\{\|\Rm\|_\infty,1\}\leq K$ for all $t\in[0,\alpha K^{-\frac{n}{2}}]$. Then
 \begin{equation*}
  \|\phi^{\frac{n}{2}(m+1)}\nabla^{\frac{n}{2}m}\Rm\|_2\leq Ct^{-\frac{m}{2}}\sup_{t\in[0,\alpha K^{-\frac{n}{2}}]}\|\Rm\|_{L^2(t),[\phi>0]}, 
 \end{equation*}
 where $C=C(m,n,\alpha,\Lambda)$, for all $t\in(0,\alpha K^{-\frac{n}{2}}]$.
\end{proposition}
\begin{proof}
 Define
 \begin{equation*}
  G(t)\equiv t^m\|\phi^{\frac{n}{2}(m+1)}\nabla^{\frac{n}{2}m}\Rm\|_2^2+\sum_{k=0}^{m-1}\beta_kt^k\|\phi^{\frac{n}{2}(k+1)}\nabla^{\frac{n}{2}k}\Rm\|_{2,[\phi>0]}^2.
 \end{equation*}
 Using Proposition \ref{prop: interpolation type estimate for patial_t nabla^k Rm},
 \begin{align*}
 \frac{dG}{dt}
 & \leq \begin{aligned}[t]
         & mt^{m-1}\|\phi^{\frac{n}{2}(m+1)}\nabla^{\frac{n}{2}m}\Rm\|_2^2 \\
         & {}+t^m\Big(\!-\tfrac{1}{2(n-2)}\|\phi^{\frac{n}{2}(m+1)}\nabla^{\frac{n}{2}(m+1)}\Rm\|_2^2+C_{\frac{n}{2}m}K^{\frac{n}{2}(m+1)}\|\Rm\|_{2,[\phi>0]}^2\Big) \\
         & {}+\sum_{k=1}^{m-1}\beta_kkt^{k-1}\|\phi^{\frac{n}{2}(k+1)}\nabla^{\frac{n}{2}k}\Rm\|_2^2 \\
         & {}+\sum_{k=0}^{m-1}\beta_kt^k\Big(\!-\tfrac{1}{2(n-2)}\|\phi^{\frac{n}{2}(k+1)}\nabla^{\frac{n}{2}(k+1)}\Rm\|_2^2+C_{\frac{n}{2}k}K^{\frac{n}{2}(k+1)}\|\Rm\|_{2,[\phi>0]}^2\Big)
        \end{aligned} \\
 & \leq \begin{aligned}[t]
         & mt^{m-1}\|\phi^{\frac{n}{2}m}\nabla^{\frac{n}{2}m}\Rm\|_2^2+t^m\Big(C_{\frac{n}{2}m}K^{\frac{n}{2}(m+1)}\|\Rm\|_{2,[\phi>0]}^2\Big) \\
         & {}+\sum_{k=0}^{m-2}\beta_{k+1}(k+1)t^k\|\phi^{\frac{n}{2}(k+1)}\nabla^{\frac{n}{2}(k+1)}\Rm\|_2^2 \\
         & {}+\sum_{k=0}^{m-1}\beta_kt^k\Big(\!-\tfrac{1}{2(n-2)}\|\phi^{\frac{n}{2}(k+1)}\nabla^{\frac{n}{2}(k+1)}\Rm\|_2^2+C_{\frac{n}{2}k}K^{\frac{n}{2}(k+1)}\|\Rm\|_{2,[\phi>0]}^2\Big).
        \end{aligned}   
 \end{align*}
If
\begin{equation*}
 mt^{m-1}\|\phi^{\frac{n}{2}m}\nabla^{\frac{n}{2}m}\Rm\|_2^2-\beta_{m-1}\tfrac{1}{2(n-2)}t^{m-1}\|\phi^{\frac{n}{2}m}\nabla^{\frac{n}{2}m}\Rm\|_2^2=0
\end{equation*}
then $\beta_{m-1}=2(n-2)m$.
If for $k$ satisfying $0\leq k\leq m-2$,
\begin{equation*}
 \beta_{k+1}(k+1)t^k\|\phi^{\frac{n}{2}(k+1)}\nabla^{\frac{n}{2}(k+1)}\Rm\|_2^2-\beta_k\tfrac{1}{2(n-2)}t^k\|\phi^{\frac{n}{2}(k+1)}\nabla^{\frac{n}{2}(k+1)}\Rm\|_2^2=0,
\end{equation*}
then
\begin{align*}
\beta_k &= 2(n-2)(k+1)\beta_{k+1} \\
& =(2n-4)^{m-k-1}(m-1)\cdots(k+1)\beta_{m-1} \\
& =(2n-4)^{m-k}m!/k!.
\end{align*}
Also define $\beta_m=1$. Using these choices for $\beta_k,0\leq k\leq m$ and choosing $t_0\in[0,\alpha K^{-\frac{n}{2}}]$ such that
\begin{equation*}
 \|\Rm\|_{L^2(t_0),[\phi>0]}=\sup_{t\in[0,\alpha K^{-\frac{n}{2}}]}\|\Rm\|_{L^2(t),[\phi>0]},
\end{equation*}
we have
\begin{align*}
 \frac{dG}{dt} &\leq\alpha^mK^{-\frac{n}{2}m}C_{\frac{n}{2}m}K^{\frac{n}{2}(m+1)}\|\Rm\|_{2,[\phi>0]}^2
 +\sum_{k=0}^{m-1}\beta_k\alpha^kK^{-\frac{n}{2}k}C_{\frac{n}{2}k}K^{\frac{n}{2}(k+1)}\|\Rm\|_{2,[\phi>0]}^2\\
 & =\sum_{k=0}^m\beta_kC_{\frac{n}{2}k}\alpha^kK^\frac{n}{2}\|\Rm\|_{2,[\phi>0]}^2 \\
 & =CK^\frac{n}{2}\|\Rm\|_{L^2(t_0),[\phi>0]}^2.
\end{align*}
Therefore
\begin{align*}
 t^m\|\phi^{\frac{n}{2}(m+1)}\nabla^{\frac{n}{2}m}\Rm\|_2^2 \leq G
 &\leq\beta_0\|\Rm\|_{L^2(0),[\phi>0]}^2+CK^\frac{n}{2}\|\Rm\|_{L^2(t_0),[\phi>0]}^2t \\
 & \leq(\beta_0+\alpha C)\|\Rm\|_{L^2(t_0),[\phi>0]}^2 \\
 & =C\|\Rm\|_{L^2(t_0),[\phi>0]}^2,
\end{align*}
proving the proposition.
\end{proof}

\begin{proposition} \label{prop: integral BBS estimate}
 Let $(M^n,g(t))$ be a solution to the AOF for $t\in[0,T)$. Let $\phi\in C_c^\infty(M)$ be a cutoff function such that
 \begin{equation*}
  \max_{0\leq i\leq\frac{n}{2}}\sup_{t\in[0,T)}\|\nabla^i\phi\|_{C^0(M,g(t))}\leq\Lambda.
 \end{equation*}
 Suppose $\max\{\|\Rm\|_{C^0(M,g(t))},1\}\leq K$ for all $t\in[0,\alpha K^{-\frac{n}{2}}]$. Then, for every $l\geq 0$ and all $t\in(0,\alpha K^{-\frac{n}{2}}]$,
 \begin{equation*}
  \|\phi^{l+\frac{n}{2}}\nabla^l\Rm\|_{L^2(M,g(t))}\leq C(1+t^{-\ceil{2l/n}/2})\sup_{t\in[0,\alpha K^{-\frac{n}{2}}]}\|\Rm\|_{L^2(\supp(\phi),g(t))}, 
 \end{equation*}
 where $C=C(l,n,\alpha,\Lambda)$.
\end{proposition}
\begin{proof}
 Let $l=\frac{n}{2}m+r,1\leq r\leq\frac{n}{2}$. Then, applying Lemma \ref{lem: nabla^l to nabla^q interpolation inequality} and Proposition \ref{prop: integral BBS estimate for multiples of n/2}, we get
 \begin{align*}
  \int_M\phi^{n(m+1)+2r}|\nabla^{\frac{n}{2}m+r}\Rm|^2
  & \leq\int_M\phi^{n(m+2)}|\nabla^{\frac{n}{2}(m+1)}\Rm|^2+C'\int_{[\phi>0]}\phi^n|\Rm|^2 \\
  & \leq t^{-(m+1)}C\Theta^2+C'\Theta^2 \\
  \|\phi^{l+\frac{n}{2}}\nabla^l\Rm\|_{L^2(t)}
  & \leq \Theta(Ct^{-\frac{m+1}{2}}+C'), 
 \end{align*}
 where
 \begin{equation*}
  \Theta=\sup_{t\in[0,\alpha K^{-\frac{n}{2}}]}\|\Rm\|_{L^2(t),[\phi>0]}.
 \end{equation*}
\end{proof}

\section{Pointwise Smoothing Estimates} \label{sec: pointwise smoothing estimates}
Let $(M,g(t))$ be a solution to AOF and let $\phi$ be a cutoff function on $M$. We give estimates of $|\nabla^i\phi|_{g(t)}$ for $1\leq i\leq\frac{n}{2}$ that depend on spacetime derivatives of the metric and $|\nabla^i\phi|_{g(0)}$ for $0\leq i\leq\frac{n}{2}$. We then give a proof of the pointwise smoothing estimates given in Theorem \ref{thm:C^m smoothing estimate along AOF}.
\begin{lemma} \label{lem: evolution of nabla^k(function)}
 Let $M$ be a manifold and $g(t)$ be a one-parameter family of metrics on $M$. For a function $\phi\in C^i(M)$ and $i\geq 2$,
 \begin{equation*}
  \partial_t\nabla^i\phi=\sum_{j=1}^{i-1}\nabla^{i-j}\partial_t g*\nabla^j\phi.
 \end{equation*}
\end{lemma}
\begin{proof}
 If $i=2$, the statement is true since via Proposition \ref{prop: commutator of partial_t and nabla^k}, we have
 \begin{equation*}
  \partial_t\nabla^2\phi=\nabla\partial_t\nabla\phi+\nabla\partial_t g*\nabla\phi=\nabla\partial_t g*\nabla\phi.
 \end{equation*}
 Suppose $i>2$ and the statement is true for every $j$ such that $2\leq j\leq i-1$. Then, using Proposition \ref{prop: commutator of partial_t and nabla^k} in the second line,
 \begin{align*}
  \partial_t\nabla^i\phi
  & =\partial_t\nabla\nabla^{i-1}\phi \\
  & =\nabla\partial_t\nabla^{i-1}\phi+\nabla\partial_t g*\nabla^{i-1}\phi \\
  & =\nabla\sum_{j=1}^{i-2}\nabla^{i-1-j}\partial_t g*\nabla^j\phi+\nabla\partial_t g*\nabla^{i-1}\phi \\
  & =\sum_{j=1}^{i-2}\nabla^{i-j}\partial_t g*\nabla^j\phi+\sum_{j=1}^{i-2}\nabla^{i-1-j}\partial_t g*\nabla^{j+1}\phi+\nabla\partial_t g*\nabla^{i-1}\phi \\
  & =\sum_{j=1}^{i-1}\nabla^{i-j}\partial_t g*\nabla^j\phi+\sum_{j=2}^{i-1}\nabla^{i-j}\partial_t g*\nabla^j\phi \\
  & =\sum_{j=1}^{i-1}\nabla^{i-j}\partial_t g*\nabla^j\phi.
 \end{align*}
\end{proof}
\begin{proposition} \label{prop: evolution of norm of nabla^k(function)}
 Let $M$ be a manifold and $g(t)$ be a one-parameter family of metrics on $M$. For a function $\phi\in C^i(M)$ and $i\geq 1$,
 \begin{equation*}
  \partial_t|\nabla^i\phi|_{g(t)}^2=\sum_{j=1}^i\nabla^{i-j}\partial_t g*\nabla^j\phi*\nabla^i\phi.
 \end{equation*}
\end{proposition}
\begin{proof}
 We compute, using the preceding Lemma \ref{lem: evolution of nabla^k(function)} in the second line:
 \begin{align*}
  \partial_t|\nabla^i\phi|_{g(t)}^2
  & =\partial_t g*\nabla^i\phi^{*2}+\partial_t\nabla^i\phi*\nabla^i\phi \\
  & =\partial_t g*\nabla^i\phi^{*2}+\sum_{j=1}^{i-1}\nabla^{i-j}\partial_t g*\nabla^j\phi*\nabla^i\phi \\
  & =\sum_{j=1}^i\nabla^{i-j}\partial_t g*\nabla^j\phi*\nabla^i\phi.
 \end{align*}
\end{proof}
\begin{proposition} \label{prop: control of cutoff function as time varies}
 Let $M$ be a Riemannian manifold with a one-parameter family of metrics $\{g(t)\}_{t\in[0,T]}$ and $\phi\in C_c^\infty(M)$. Fix $i\geq 1$. Suppose that, for each $j$ satisfying $0\leq j\leq i-1$, there exists $K_j>0$ such that $|\nabla^j\partial_t g(x,t)|_{g(t)}\leq K_j$ on $\supp\phi\x[0,T]$ and, for each $j$ satisfying $1\leq j\leq i$, there exists $C_j'>0$ such that $|\nabla^j\phi|_{g(0)}\leq C_j'$ on $\supp\phi$. Then there exists a constant $C_i$ such that
 \begin{equation*}
  |\nabla^i\phi|_{g(t)}^2\leq C_i=C_i(K_0,\dots,K_{i-1},C_1',\dots,C_i',T).
 \end{equation*}
\end{proposition}
\begin{proof}
 Let $i=1$. Then Proposition \ref{prop: evolution of norm of nabla^k(function)} gives
 \begin{equation*}
  \partial_t|\nabla\phi|_{g(t)}^2=\partial_t g*\nabla\phi^{*2}\leq CK_0|\nabla\phi|_{g(t)}^2.
 \end{equation*}
 Solving the differential inequality, we get
 \begin{equation*}
  |\nabla\phi|_{g(t)}^2\leq|\nabla\phi|_{g(0)}^2e^{CK_0T}\equiv C_1^2
 \end{equation*}
 which proves the proposition for $i=1$.
 
 Fix $i\geq 2$ and suppose that the proposition is true for every $j$ satisfying $1\leq j\leq i-1$. Let $f(t)=|\nabla^i\phi|_{g(t)}^2$. Then, via Proposition \ref{prop: evolution of norm of nabla^k(function)},
 \begin{align*}
  \frac{df}{dt}
  & \leq\sum_{j=1}^i|\nabla^{i-j}\partial_t g*\nabla^j\phi*\nabla^i\phi| \\
  & \leq\sum_{j=1}^{i-1}|\nabla^{i-j}\partial_t g||\nabla^j\phi||\nabla^i\phi|+|\partial_t g||\nabla^i\phi|^2 \\
  & \leq\sum_{j=1}^{i-1}CK_{i-j}C_jf^\frac{1}{2}+CK_0f \\
  & \leq\widetilde{C}(K_0,\dots,K_{i-1},C_1,\dots,C_{i-1})(1+f) \\
  & =\widetilde{C}(K_0,\dots,K_{i-1},C_1',\dots,C_{i-1}',T)(1+f).
 \end{align*}
 Solving the differential inequality, we get
 \begin{align*}
  1+f(t) & \leq(1+f(0))\widetilde{C}t \\
  |\nabla^i\phi|_{g(t)}^2 & \leq (1+|\nabla^i\phi|_{g(0)}^2)e^{\widetilde{C}T} \\
  & \leq (1+(C_i')^2)e^{\widetilde{C}T}\equiv C_i^2.
 \end{align*}
\end{proof}

\begin{proposition} \label{prop: control of nabla^k Rm using cutoff function at one time}
 Let $(M^n,g(t))$ solve AOF on $[0,T]$, where $n\geq 4$. Fix $r>0$. Suppose there exist $x\in M$, $r>0$, and $K>0$ such that
\begin{equation}  \label{eqn: control of nabla^k Rm using cutoff function at one time: curvature bound}
\max\left[1,\sup_{[0,T]}\|\Rm\|_{C^0(B_{g(T)}(x,2r),g(t))}\right]+\sum_{j=1}^{3n/2-3}\sup_{[0,T]}\|\Rm\|^{\frac{2}{j+2}}_{C^0(B_{g(T)}(x,2r),g(t))}<K.
\end{equation}
 Then for all $l\geq 0$ and $t\in(0,T]$,
 \begin{equation} \label{eqn: control of nabla^k Rm using cutoff function at one time: desired bound}
  \|\nabla^l\Rm\|_{L^2(B_{g(T)}(x,r),g(t))}\leq C(1+t^{-\ceil{2l/n}/2})\sup_{t\in[0,T]}\|\Rm\|_{L^2(B_{g(T)}(x,2r),g(t))}, 
 \end{equation}
 where $C=C(n,l,K,T,\frac{1}{r})$.
\end{proposition}
\begin{proof}
 Let $\phi$ be a cutoff function that is equal to $1$ on $B_{g(T)}(x,r)$ and supported on $B_{g(T)}(x,2r)$. The inequality \eqref{eqn: control of nabla^k Rm using cutoff function at one time: curvature bound} provides $C^0$ bounds for the first $\frac{n}{2}-2$ covariant derivatives of $\Rm$, so that
 \begin{equation} \label{eqn: control of nabla^k Rm using cutoff function at one time: cutoff bound}
  \max_{0\leq j\leq\frac{n}{2}}\|\nabla^j\phi\|_{C^0(M,g(T))}\leq C'(n,K,\tfrac{1}{r}).
 \end{equation}
The inequality \eqref{eqn: control of nabla^k Rm using cutoff function at one time: cutoff bound} provides bounds for the first $\frac{n}{2}$ covariant derviatives of $\phi$ at time $T$, and the inequality \eqref{eqn: control of nabla^k Rm using cutoff function at one time: curvature bound} indudes bounds on the first $\frac{n}{2}-1$ covariant derivatives of $\Obhat$. We therefore are able to, for each $t\in[0,T]$ and $j$ satisfying $0\leq j\leq\frac{n}{2}$, to obtain via Proposition \ref{prop: control of cutoff function as time varies} bounds given by
 \begin{equation*}
  \|\nabla^j\phi\|_{C^0(M,g(t))}\leq\widetilde{C}_j(n,K,\tfrac{1}{r},T).
 \end{equation*}
 Therefore, via Proposition \ref{prop: integral BBS estimate},
 \begin{align*}
  \|\nabla^l\Rm\|_{L^2(B_{g(T)}(x,r),g(t))}
  & \leq\|\phi^{l+\frac{n}{2}}\nabla^l\Rm\|_{L^2(M,g(t))} \\
  & \leq C(1+t^{-\ceil{2l/n}/2})\sup_{t\in[0,T]}\|\Rm\|_{L^2(\supp(\phi),g(t))} \\
  & = C(1+t^{-\ceil{2l/n}/2})\sup_{t\in[0,T]}\|\Rm\|_{L^2(B_{g(T)}(x,2r),g(t))}
 \end{align*}
 where $C=C(n,l,K,T,\frac{1}{r})$.
\end{proof}
We are now able to prove the pointwise smoothing estimates given in Theorem \ref{thm:C^m smoothing estimate along AOF}.
\begin{proof}[Proof of Theorem \ref{thm:C^m smoothing estimate along AOF}]
 We adapt the proof of Theorem 1.3 in Streets \cite{StreetsLongTimeBehaviorOfFourthOrderCurvatureFlows}. We will show that if this inequality fails, we can construct a blowup limit that is flat and has nonzero curvature. Consider the function given by
 \begin{equation*}
  f_m(x,t,g)=\sum_{j=1}^m|\nabla^j\Rm(g(x,t))|_{g(t)}^{\frac{2}{j+2}}.
 \end{equation*}
 It suffices to show that
 \begin{equation} \label{eqn:C^m smoothing estimate along AOF:f_m inequality}
   f_m(x,t,g)\leq C\left(K+\frac{1}{t^\frac{2}{n}}\right) 
 \end{equation}
 since for every $l$ satisfying $1\leq l\leq m$,
 \begin{equation*}
  |\nabla^l\Rm(g(x,t))|_{g(t)}^{\frac{2}{l+2}}\leq\sum_{j=1}^m|\nabla^j\Rm(g(x,t))|_{g(t)}^{\frac{2}{j+2}}=f_m(x,t,g)\leq C\left(K+\frac{1}{t^\frac{2}{n}}\right)
 \end{equation*}
 and
 \begin{equation*}
  |\nabla^l\Rm(g(x,t))|_{g(t)}\leq C\left(K+\frac{1}{t^\frac{2}{n}}\right)^{\frac{l+2}{2}}\leq C\left(K+\frac{1}{t^\frac{2}{n}}\right)^{\frac{m+2}{2}}.
 \end{equation*}
 Suppose that the inequality \eqref{eqn:C^m smoothing estimate along AOF:f_m inequality} fails. It suffices to take $m\geq\frac{3n}{2}-3$. Without loss of generality, for each $i\in\N$ there exists a solution to AOF $(M_i^n,g_i(t))$ and $(x_i,t_i)\in M_i\x(0,T]$ such that
 \begin{equation*}
  i<\frac{f_m(x_i,t_i,g_i)}{K+t_i^{-\frac{2}{n}}}=\sup_{M_i\x(0,T]}\frac{f_m(x,t,g_i)}{K+t^{-\frac{2}{n}}}<\infty.
 \end{equation*}
 and define a new sequence of blown up metrics by
 \begin{equation*}
  \widetilde{g}_i(t)=\lambda_ig_i(t_i+\lambda_i^{-\frac{n}{2}}t),
 \end{equation*}
 where $\lambda_i=f_m(x_i,t_i,g_i)$.
 We will show in the next section that these metrics also solve AOF. These metrics, which are defined for $t\in[-\lambda_i^\frac{n}{2}t_i,0]$, are eventually defined on $[-1,0]$ since as $i\to\infty$,
 \begin{equation*}
  t_i^{\frac{2}{n}}\lambda_i=\frac{f_m(x_i,t_i,g_i)}{t_i^{-\frac{2}{n}}}\geq\frac{f_m(x_i,t_i,g_i)}{K+t_i^{-\frac{2}{n}}}\to\infty.
 \end{equation*}
 Replace the sequence of AOF solutions $\{(M_i,\widetilde{g}_i(t))\}_{i\in\N}$ with the tail subsequence for which $\lambda_i^{\frac{n}{2}}t_i>1$. The curvatures of these manifolds converge to $0$ since as $i\to\infty$,
 \begin{equation} \label{eqn: :C^m smoothing estimate along AOF: curv conv to 0}
  |\Rm(\widetilde{g}_i)|_{\tilde{g}_i}\leq\frac{K}{\lambda_i}=\frac{K}{f_m(x_i,t_i,g_i)}\leq\frac{K+t_i^{-\frac{2}{n}}}{f_m(x_i,t_i,g_i)}\to 0.
 \end{equation}

 Furthermore, there is a uniform $C^m$ estimate on the curvature given by
 \begin{align} \label{eq:C^m smoothing estimate along AOF: blowup curvature bound}
  f_m(x,t,\widetilde{g}_i)
  & = \frac{f_m(x_i,t_i+t\lambda_i^{-\frac{n}{2}},g_i)}{\lambda_i} \nonumber \\
  & = \frac{f_m(x_i,t_i+t\lambda_i^{-\frac{n}{2}},g_i)}{f_m(x_i,t_i,g_i)} \nonumber \\
  & \leq \frac{K+(t_i+t\lambda_i^{-\frac{n}{2}})^{-\frac{2}{n}}}{K+t_i^{-\frac{2}{n}}} \nonumber \\
  & \leq \frac{K+t_i^{-\frac{2}{n}}(1+\tfrac{t}{2})^{-\frac{2}{n}}}{K+t_i^{-\frac{2}{n}}} \nonumber \\
  & \leq 2^{\frac{2}{n}}
 \end{align}
 for all $i\in\N$ and $(x,t)\in M_i\x[-1,0]$. 

 Let $\phi_i:B(0,1)\to M_i$ be given by $\exp_{x_i}$ with respect to $g_i(0)$ for each $i\in\N$ and $h_i(t)\equiv\phi_i^*g_i(t)$. The uniform $C^0$ bound on $\Rm(\widetilde{g_i}(t))$ given by \eqref{eq:C^m smoothing estimate along AOF: blowup curvature bound} induces a uniform bound on $(\phi_i)_*$ (see Petersen \cite{PetersenRiemannianGeometry}) which permits the uniform $C^m$ estimate \eqref{eq:C^m smoothing estimate along AOF: blowup curvature bound} on $\Rm(\widetilde{g_i}(t))$ to lift to a  uniform $C^m$ estimate on  $\Rm(h_i(t))$. Furthermore, $h_i(t)$ solves AOF for all $i$ since $\phi_i$ does not depend on $t$.

Since $m\geq\frac{3n}{2}-3$, we have uniform $C^0$ bounds on $\nabla^j\Obhat(g(t))$ for $0\leq j\leq\frac{n}{2}-1$. Via Proposition \ref{prop: control of nabla^k Rm using cutoff function at one time}, we obtain uniform bounds on the $L^2(B_{h_i(0)}(0,\frac{1}{2}))$-norms of all covariant derivatives of $\Rm(h_i(0))$. Since the metrics $h_i(0)$ are uniformly equivalent to the Euclidean metric, the Sobolev constant of $B_{h_i(0)}(0,\frac{1}{2})$ is uniformly bounded for all $i$. Via the Kondrakov compactness theorem, we thus obtain uniform bounds on the $C^0(B_{h_i(0)}(0,\frac{1}{2}))$-norms of all covariant derivatives of $\Rm(h_i(0))$. The Taylor expansion for $h_i$ in terms of geodesic coordinates about $0$ with curvature coefficients can then be used to obtain uniform bounds on the $C^0(B_{h_i(0)}(0,\frac{1}{2}))$-norms of all covariant derivatives of $h_i(0)$. Finally, by the Arzel\`{a}-Ascoli - type Proposition \ref{prop: arzela-ascoli for metrics}, after taking a subsequence, still named $\{h_i(0)\}_{i\in\N}$, we get $h_i(0)\to h_\infty$ in $C^\infty(B(0,\frac{1}{2}))$ for some Riemannian metric $h_\infty$. We have already shown with inequality \eqref{eqn: :C^m smoothing estimate along AOF: curv conv to 0} that $(B(0,\frac{1}{2}),h_\infty)$ is flat. However, for all $i\in\N$,
 \begin{align*}
  f_m(x_i,0,g_i)
  & =\sum_{j=1}^m|\nabla_{\widetilde{g}_i}^j\Rm(\widetilde{g}_i)(x_i,0)|_{\tilde{g}_i(0)}^{\frac{2}{2+j}} \\
  & =\sum_{j=1}^m\left(\lambda_i^{-\frac{j+2}{2}}|\nabla^j\Rm(x_i,t_i)|_{g(t_i)}\right)^{\frac{2}{2+j}} \\
  & =\sum_{j=1}^m\lambda_i^{-1}|\nabla^j\Rm(x_i,t_i)|_{g(t_i)}^{\frac{2}{2+j}} \\
  & =\lambda_i^{-1}\lambda_i=1.
 \end{align*}
Also, $f_m(0,0,h_i)=1$ for all $i$ since $(\phi_i)_*$ is the identity map at $0$. Therefore $f_m(0,0,h_\infty)=1$. This is a contradiction, thereby proving the the inequality \eqref{eqn:C^m smoothing estimate along AOF:f_m inequality}.
\end{proof}

\section{Long Time Existence} \label{sec: long time existence}
In this section, we prove that if a solution $(M,g(t))$ to the AOF only exists for a finite time $T$, then $\|\Rm\|_\infty$ becomes unbounded along a sequence $\{(x_n,t_n)\}_{n=1}^\infty\subset M\times[0,T)$ with $t_n\uparrow T$.
We will prove this theorem by showing that if actually
\begin{equation}
 \sup_{t\in[0,T)}\|\Rm\|_{C^0(g(t))}=K<\infty, \label{eq: |Rm|_g(t) bdd on M x [0,T)}
\end{equation}
then the solution $g(t)$ exists past the time $T$. In order to show this, we show that \eqref{eq: |Rm|_g(t) bdd on M x [0,T)} and the pointwise smoothing estimates on $|\nabla^k\Rm|_{g(t)}$ induce bounds on $|\bar{\nabla}^k g(t)|_{\bar{g}}$ with respect to some fixed background metric $\bar{g}$ and connection $\bar{\nabla}$. We also show that \eqref{eq: |Rm|_g(t) bdd on M x [0,T)} implies uniform convergence of $g(t)$ to some continuous metric $g(T)$. The bounds on $|\bar{\nabla}^k g(t)|_{\bar{g}}$ imply that $g(T)$ is smooth, so that we can extend the solution $g(t)$ past the time $T$ via the short time existence theorem \ref{thm: short time existence}.

We first show that if \eqref{eq: |Rm|_g(t) bdd on M x [0,T)} holds, the metrics $g(t)$ converge uniformly as $t\uparrow T$ to a continuous metric $g(T)$ equivalent to each $g(t)$. The following lemma is from Chow-Knopf \cite{ChowKnopfRicciFlowIntro}:
\begin{lemma} \label{lem: uniform convergence of metrics satisfying differential inequality on [0,T)}
 Let $M$ be a closed manifold. For $0\leq t<T\leq\infty$, let $g(t)$ be a one-parameter family of metrics on $M$ depending smoothly on both space and time. If there exists a constant $C<\infty$ such that
 \begin{equation*}
  \int_0^T\left|\frac{\partial}{\partial t}g(x,t)\right|_{g(t)}\intd t\leq C
 \end{equation*}
 for all $x\in M$, then
 \begin{equation*}
  e^{-C}g(x,0)\leq g(x,t)\leq e^Cg(x,0)
 \end{equation*}
 for all $x\in M$ and $t\in[0,T)$. Furthermore, as $t\uparrow T$, the metrics $g(t)$ converge uniformly to a continuous metric $g(T)$ such that for all $x\in M$,
 \begin{equation*}
  e^{-C}g(x,0)\leq g(x,T)\leq e^Cg(x,0).
 \end{equation*}
\end{lemma}
\begin{lemma} \label{lem: uniform convergence of metrics along AOF on [0,T)}
 Let $M$ be a compact manifold and let $(M,g(t))$ be a solution to AOF on $[0,T)$ such that
\begin{equation*}
\sup_{t\in[0,T)}\|\Rm\|_{C^0(g(t))}=K<\infty.
\end{equation*}
Then $g(t)$ converges uniformly as $t\uparrow T$ to a continuous metric $g(T)$ that is uniformly equivalent to $g(t)$ for every $t\in[0,T]$. 
\end{lemma}
\begin{proof}
 Since Proposition \ref{prop: Obhat expanded} states that
 \begin{equation*}
  \frac{\dif g}{\dif t}
  =\frac{(-1)^{\frac{n}{2}}}{n-2}\Delta^{\frac{n}{2}-1}\Rc+\frac{(-1)^{\frac{n}{2}-1}}{2(n-1)}\Delta^{\frac{n}{2}-2}\nabla^2 R+\sum_{j=2}^{n/2}P_j^{n-2j}(\Rm),
 \end{equation*}
 in order to apply the preceding Lemma \ref{lem: uniform convergence of metrics satisfying differential inequality on [0,T)} it suffices to show that $|\nabla^k\Rm|_{g(t)}$ is bounded on $M\times[0,T)$ for al $k$ satisfying $0\leq k\leq n-2$. Using the smoothng estimate provided in Theorem \ref{thm:C^m smoothing estimate along AOF}, we get
\begin{equation*}
\max_{0\leq k\leq n-2}\sup_{M\x[0,T)}|\nabla^k\Rm|_{g(t)}\leq\max_{0\leq k\leq n-2}\sup_{M\x[0,\frac{T}{2}]}|\nabla^k\Rm|_{g(t)}+C\big(\widetilde{K}+(\tfrac{T}{2})^{-\frac{2}{n}}\big)^{\frac{n}{2}},
\end{equation*}
where $C=C(n)$ and $\widetilde{K}=\max\{K,1\}$.

So $\frac{\dif g}{\dif t}$ is bounded on $M\times[0,T)$ and the metrics $g(t)$ converge uniformly as $t\uparrow T$ to a continuous metric $g(T)$ uniformly equivalent to each $g(t)$.
\end{proof}

Since $M$ is a compact manifold, we can obtain bounds on $|\bar{\nabla}^k g(t)|_{\bar{g}}$ by taking the maximum of bounds taken on finitely many coordinate patches. On such a coordinate patch, we can assume that the fixed metric is just the Euclidean one. Thus we will only need to bound the partial derivatives of $g$ and $\Obhat$.

\begin{lemma} \label{lem: partial derivatives of g and Obhat bdd wrt covariant derivatives of Rm}
  Let $M$ be a compact manifold and let $(M,g(t))$ be a solution to AOF on $[0,T)$. Fix $m\geq 0$. Suppose that for $0\leq i\leq m+n-1$, there exist constants $C_i$ such that $|\nabla_{g(t)}^i\Rm(g(t))|_{g(t)}<C_i$ on $M\x[0,T)$. Then for all $t\in[0,T)$,
  \begin{align*} 
   |\partial^m g(t)|_{g(t)} & <\widetilde{C}_1(g(0),C_0,\dots,C_{m+n-1}) \\
   |\partial^m \Obhat(t)|_{g(t)} & <\widetilde{C}_2(g(0),C_0,\dots,C_{m+n-1}).
  \end{align*}
\end{lemma}
\begin{proof}
 We prove this by induction. First we bound $\dif g$. We have
 \begin{equation*}
  \dif_t\dif g=\dif\dif_t g=(\nabla+\Gamma)\ast\dif_t g=\nabla\Obhat+\Gamma\ast\Obhat.
 \end{equation*}
 From the definition of $\Obhat$, we obtain the bound $\nabla\Obhat<C(C_0,\dots,C_{n-1})$. Then, since $\dif_t\Gamma=\nabla\dif_t g=\nabla\Obhat$, $\Gamma$ can be bounded in terms of the inital metric and $\nabla\Obhat$ after integrating. So $\dif\Obhat=\dif_t\dif g$ is uniformly bounded by $C(g(0),C_0,\dots,C_{n-1})$, and so is $\dif g$ after integrating.
 
 Assume that
 \begin{align*}
  |\partial^i g| & <C(g(0),C_0,\dots,C_{i+n-1})\textrm{ for }0\leq i\leq m-1, \\
  |\partial^i\Obhat| & <C(g(0),C_0,\dots,C_{i+n-1})\textrm{ for }0\leq i\leq m-1, \\
  |\partial^i\Gamma| & <C(g(0),C_0,\dots,C_{i+n-1})\textrm{ for }0\leq i\leq m-2.
 \end{align*}
 We wish to bound $\dif^m g$. It suffices to bound $\dif^m\Obhat$ since $\dif_t\dif^m g=\dif^m\dif_t g=\dif^m\Obhat$. We can express $\dif^m\Obhat$ as
 \begin{equation} \label{eq: partial(Obhat) in terms of nabla(Obhat) and Gamma}
  \dif^m\Obhat=\nabla^m\Obhat+\sum_{i=0}^{m-1}\dif^i\Obhat\ast\mathsf{P}^{m-i}(\Gamma),
 \end{equation}
 where $\mathsf{P}^k(A)$ is defined to be some polynomial in $A$ such that for each term the sum of the number of partial derivatives of $g$ in each factor is at most $k$. The following is a proof by induction. First, the equation holds when $m=1$: $\dif\Obhat=(\nabla+\Gamma)\ast\Obhat=\nabla\Obhat+\Gamma\ast\Obhat$.
 Assume the equation \eqref{eq: partial(Obhat) in terms of nabla(Obhat) and Gamma} holds for $0\leq i\leq m$. Then
 \begin{align*}
  \nabla^{m+1}\Obhat &= (\dif+\Gamma)\nabla^m\Obhat\\
  &= \partial^{m+1}\Obhat+\partial^m\Obhat*\mathsf{P}^1(\Gamma)+\sum_{i=0}^{m-1}\left[\dif^{i+1}\Obhat\ast\mathsf{P}^{m-i}(\Gamma)+\dif^i\Obhat\ast\mathsf{P}^{m+1-i}(\Gamma)+\Gamma\ast\dif^i\Obhat\ast\mathsf{P}^{m-i}(\Gamma)\right]\\
  &= \partial^{m+1}\Obhat+\partial^m\Obhat*\mathsf{P}^1(\Gamma)+\sum_{i=0}^{m-1}\left[\dif^{i+1}\Obhat\ast\mathsf{P}^{m-i}(\Gamma)+\dif^i\Obhat\ast\mathsf{P}^{m+1-i}(\Gamma)\right]\\
  &= \partial^{m+1}\Obhat+\partial^m\Obhat*\mathsf{P}^1(\Gamma)+\sum_{i=1}^m\dif^i\Obhat\ast\mathsf{P}^{m+1-i}(\Gamma)+\sum_{i=0}^{m-1}\dif^i\Obhat\ast\mathsf{P}^{m+1-i}(\Gamma)\\
  &= \partial^{m+1}\Obhat+\sum_{i=0}^m\dif^i\Obhat\ast\mathsf{P}^{m+1-i}(\Gamma).
 \end{align*}
 From the equation \eqref{eq: partial(Obhat) in terms of nabla(Obhat) and Gamma}, we see that in order to bound $\dif^m\Obhat$, we only need to bound $\dif^{m-1}\Gamma$. We have
 \begin{equation} \label{eq: partial_t partial^(m-1) Gamma}
  \dif_t\dif^{m-1}\Gamma=\dif^{m-1}\dif_t\Gamma=\sum_{i=0}^{m-1}\dif^i\nabla\Obhat.
 \end{equation}
 We bound $\dif^i\nabla\Obhat$ via the equation
 \begin{equation} \label{eq: partial nabla(Obhat)}
  \dif^i\nabla\Obhat=\nabla^{i+1}\Obhat+\sum_{j=1}^i\nabla^j\Obhat\ast\mathsf{P}^{i-j+1}(\Gamma).
 \end{equation}
 In order to verify this via induction, we have that for $i=1$, $\dif\nabla\Obhat=\nabla^2\Obhat+\Gamma\ast\nabla\Obhat.$ If the equation holds for the $i$th partial derivative,
 \begin{align*}
  \dif^{i+1}\nabla\Obhat &= (\nabla+\Gamma)\dif^i\nabla\Obhat\\
  &= \nabla^{i+2}\Obhat+\sum_{j=1}^i\left[\nabla^{j+1}\Obhat\ast\mathsf{P}^{i-j+1}(\Gamma)+\nabla^j\Obhat*\mathsf{P}^{i-j+2}(\Gamma)\right]+\Gamma\ast\left[\nabla^{i+1}\Obhat+\sum_{j=1}^i\nabla^j\Obhat\ast\mathsf{P}^{i-j+1}(\Gamma)\right]\\
  &= \nabla^{i+2}\Obhat+\sum_{j=2}^{i+1}\nabla^j\Obhat\ast\mathsf{P}^{i-j+2}(\Gamma)+\sum_{j=1}^i\nabla^j\Obhat\ast\mathsf{P}^{i-j+2}(\Gamma)+\nabla^{i+1}\Obhat\ast\mathsf{P}^1(\Gamma)+\sum_{j=1}^i\nabla^j\Obhat\ast\mathsf{P}^{i-j+2}(\Gamma)\\
  &= \nabla^{i+2}\Obhat+\sum_{j=1}^{i+1}\nabla^j\Obhat\ast\mathsf{P}^{i-j+2}(\Gamma).
 \end{align*}
 If $0\leq i\leq m-1$, then the highest partial derivative of $\Gamma$ that appears in equation \eqref{eq: partial nabla(Obhat)} is of order at most $m-2$, so $\dif^i\nabla\Obhat$ is bounded in terms of covariant derivatives of $\Obhat$ and previously bounded partial derivatives of $\Gamma$. Therefore, via equation \eqref{eq: partial_t partial^(m-1) Gamma}, $\dif^{m-1}\Gamma$ and $\dif^m\Obhat$ are bounded.
\end{proof}
\begin{proof}[Proof of Theorem  \ref{thm: obstruction to long time existence}]
 Suppose that equation \eqref{eq: |Rm|_g(t) bdd on M x [0,T)} holds. By Lemma \ref{lem: uniform convergence of metrics along AOF on [0,T)}, the metrics $g(t)$ converge uniformly to a continuous metric $g(T)$ as $t\uparrow T$.  We show that $g(T)$ is $C^\infty$ on $M$. It suffices to show for each $k\in\N$ that $g(T)$ is $C^k$ on any coordinate patch since we can take a maximum over finitely many of them to show that $g(T)$ is $C^k$ on $M$. We have
 \begin{equation*}
  g(t)=g(0)+\int_0^t\Obhat(\tau)\,d\tau.
 \end{equation*}
 Taking limits as $t\uparrow T$, we get
 \begin{equation*}
  g(T)=g(0)+\int_0^T\Obhat(\tau)\,d\tau.
 \end{equation*}
 This permits us to take the $k$th partial derviative:
 \begin{equation*}
  \partial^k g(T)=\partial^kg(0)+\int_0^T\partial^k\Obhat(\tau)\,d\tau.
 \end{equation*}
 The bounds on $\dif^kg$ and $\dif^k\Obhat$ from Lemma \ref{lem: partial derivatives of g and Obhat bdd wrt covariant derivatives of Rm} therefore imply a bound on $\dif^kg(T)$. So $g(T)$ is $C^\infty$ on $M$. Furthermore, since
 \begin{equation*}
  |\dif^k g(T)-\dif^k g(t)|\leq\int_t^T|\dif^k\Obhat(\tau)|\,d\tau\leq C_k(T-t),
 \end{equation*}
 the metrics $g(t)$ converge in $C^\infty$ to $g(T)$. So $g(t)$ is a $C^\infty$ solution to AOF on $[0,T]$. Then the short time existence Theorem \ref{thm: short time existence} applied to $g(t)$ with inital metric $g(T)$ allows us to extend $g(t)$ past $T$. This contradicts the assumption that $T$ was the maximal time for the solution $(M,g(t))$.
\end{proof}

\section{Compactness of Solutions} \label{sec: compactness of solutions}
In this section, we give compactness results for AOF similar to Hamilton's compactness theorem for solutions of the Ricci flow. We first prove a proposition that states that for a sequence of metrics,  uniform bounds on the spacetime derivatives of curvature and the derivatives of the metric at one time extend to uniform bounds on the spacetime derivatives of the metric. This is used to prove the compactness Theorem \ref{thm: cptness for seq of cpt mfds} for a sequence of complete pointed solutions of AOF. We then give the proofs of Theorem \ref{thm: existence of singularity model at curvature singularity}, which allows us to obtain a singularity model from a singular solution, and Theorem \ref{thm: nonsingular limit at infinity}, which describes the behavior at time $\infty$ of a nonsingular solution.

We quote some definitions and results from Chow et. al.'s text \cite{ChowEtAlRicciFlowTechniquesApplications} on Ricci flow.
\begin{definition}
(\cite{ChowEtAlRicciFlowTechniquesApplications} Definition 3.1) Let $K\subset M$ be a compact set and let $\{g_k\}_{k\in\N}$, $g_\infty$, and $g$ be Riemannian metrics on $M$. For $p\in\{0\}\cup\N$ we say that $g_k$ \textbf{converges in $C^p$ to $g_\infty$ uniformly on $K$} if for every $\epsilon>0$ there exists $k_0=k_0(\epsilon)$ such that for $k\geq k_0$,
 \begin{equation*}
  \sup_{0\leq|\alpha|\leq p}\sup_{x\in K}|\nabla_g^\alpha(g_k-g_\infty)|_g<\epsilon.
 \end{equation*}
\end{definition}
\begin{definition}
(\cite{ChowEtAlRicciFlowTechniquesApplications} Definition 3.5) ($C^\infty$ Cheeger-Gromov convergence) A sequence $\{(M_k^n,g_k,O_k)\}_{k\in\N}$ of complete pointed Riemannian manifolds \textbf{converges} (in the Cheeger-Gromov topology) to a complete pointed Riemannian manifold $(M_\infty^n,g_\infty,O_\infty)$ if there exist
\begin{enumerate}
\item an exhaustion $\{U_k\}_{k\in\N}$ of $M_\infty$ by open sets with $O_\infty\in U_k$,
\item a sequence of diffeomorphisms $\Phi_k:U_k\to V_k:=\Phi_k(U_k)\subset M_k$ with $\Phi_k(O_\infty)=O_k$, such that $\left(U_k,\Phi_k^*\left[g_k|_{V_k}\right]\right)$ converges in $C^\infty$ to $(M_\infty,g_\infty)$ uniformly on compact sets in $M_\infty$.
\end{enumerate}
\end{definition}
\begin{definition}
(\cite{ChowEtAlRicciFlowTechniquesApplications} Definition 3.6) A sequence $\{(M_k^n,g_k(t),O_k)\}_{k\in\N}$ of one-parameter families of complete pointed Riemannian manifolds \textbf{converges} to a one-parameter family of complete pointed Riemannian manifolds $(M_\infty^n,g_\infty(t),O_\infty)$, $t\in(\alpha,\omega)$, if there exist
\begin{enumerate}
\item  an exhaustion $\{U_k\}_{k\in\N}$ by open sets with $O_\infty\in U_k$
\item a sequence of diffeomorphisms $\Phi_k:U_k\to V_k:=\Phi_k(U_k)\subset M_k$ with $\Phi_k(O_\infty)=O_k$, such that $\left(U_k\times(\alpha,\omega),\Phi_k^*\left[g_k(t)|_{V_k}\right]+dt^2\right)$ converges in $C^\infty$ to $ (M_\infty\times(\alpha,\omega),g_\infty(t)+dt^2)$ uniformly on compact subsets in $M_\infty\times(\alpha,\omega)$.
\end{enumerate}
\end{definition}
\begin{theorem} \label{thm: Cheeger Gromov cptness}
 (Cheeger-Gromov compactness theorem) (Hamilton, \cite{HamiltonCptnessPropertyForSolutionsOfTheRicciFlow} Theorem 2.3) Let $\{(M_k^n,g_k,O_k)\}_{k\in\N}$ be a sequence of complete pointed Riemannian manifolds that satisfy
 \begin{equation*}
  |\nabla_k^p\Rm_k|_k\leq C_p\textrm{ on }M_k
 \end{equation*}
 for all $p\geq 0$ and $k$ where $C_p<\infty$ is a sequence of constants independent of $k$ and
 \begin{equation*}
  \mathrm{inj}_{g_k}(O_k)\geq\iota_0
 \end{equation*}
 for some constant $\iota_0>0$. Then there exists a subsequence $\{j_k\}_{k\in\N}$ such that $\{M_{j_k},g_{j_k},O_{j_k})\}_{k\in\N}$ converges to a complete pointed Riemannian manifold $(M_\infty^n,g_\infty,O_\infty)$ as $k\to\infty$.
\end{theorem}
The following proposition allows us to extend bounds on the derivatives of a sequence of metrics at one time to bounds that are uniform over an interval.
\begin{proposition} \label{prop: extend convergence of metrics over time interval}
 Let $(M,g)$ be a Riemannian manifold and $L$ be a compact subset of $M$. Let $\{g_i\}_{i\in\N}$ be a collection of Riemannian metrics that are solutions of AOF on neighborhoods containing $L\times[\beta,\psi]$. Let $t_0\in[\beta,\psi]$ and fix $k\geq n-2$. Let unmarked objects such as $\nabla$ and $|\cdot|$ be taken with respect to $g$, and let objects such as $\nabla_k$ and $|\cdot|_k$ be taken with respect to $g_k$.  Suppose that:
 \begin{enumerate}
  \item The metrics $g_i(t_0)$ are uniformly equivalent to $g$ for every $i\in\N$: for some $B_0>0$, $B_0^{-1}g\leq g_i(t_0)\leq B_0g.$
  \item For each $1\leq p\leq k$, there exists a uniform bound $C_p$ on $L$ independent of $i$ such that $|\nabla^p g_i(t_0)|\leq C_p$.
  \item For each $0\leq p+q\leq k+n-2$, there exists a uniform bound $C'_{p,q}$ on $L\times[\beta,\psi]$ independent of $i$ such that $|\partial_t^q\nabla_{g_i}^p\Rm(g_i)|_{g_i}\leq C'_{p,q}$.
 \end{enumerate}
 Then:
 \begin{enumerate}
  \item The metrics $g_k(t)$ are uniformly equivalent to $g$ for every $i\in\N$ and $t\in[\beta,\phi]$: for some $B=B(t,t_0)>0$, $B^{-1}g\leq g_i(t)\leq Bg$.
  \item For every $p,q$ satisfying $0\leq p+q\leq k$, there is a uniform bound $\widetilde{C}_{p,q}$ on $L\times[\beta,\psi]$ independent of $i$ such that $|\dif_t^q\nabla^p g_i(t)|\leq\widetilde{C}_{p,q}$.
 \end{enumerate}
\end{proposition}
\begin{lemma} \label{lem: extend convergence of metrics over time interval: uniform spacetime equivalence of metrics}
 The metrics $g_k(t)$ in the above proposition are uniformly equivalent to $g$ on $L\times [\beta,\psi]$: for all $V\in T_xL$ with $x\in L$,
 \begin{equation*}
  B(t,t_0)^{-1}g(V,V)\leq g_k(t)(V,V)\leq B(t,t_0)g(V,V).
 \end{equation*}
\end{lemma}
\begin{proof}
 We show that $\left|\frac{\dif}{\dif t}\log g_k(t)(V,V)\right|$ is bounded uniformly in $k$. Fix $k\in\N$. First,
 \begin{equation*}
  \left|\frac{\dif}{\dif t}\log g_k(t)(V,V)\right|=\left|\frac{\frac{\dif}{\dif t}g_k(t)(V,V)}{g_k(t)(V,V)}\right|.
 \end{equation*}
 Since the numerator and denominator are bilinear, it suffices to show the above is bounded when $g_k(V,V)=1$, in which case the right hand side reduces to $|\dif_tg_k(t)(V,V)|$.
 In order to show this is bounded, we use the flow equation \eqref{eq: short time existence system} and the expression for the gradient given by \eqref{eq: gradient expressed with covariant derivatives of Rm}.
 \begin{align*}
  \left|\frac{\dif}{\dif t}g_k(t)(V,V)\right| & \leq\left|\frac{\dif}{\dif t}g_k(t)\right|_k \\
  & =\left|\frac{(-1)^{\frac{n}{2}}}{n-2}\Delta_k^{\frac{n}{2}-1}\Rc_k+\frac{(-1)^{\frac{n}{2}-1}}{2(n-1)}\Delta_k^{\frac{n}{2}-2}\nabla_k^2 R_k+\sum_{j=2}^{n/2}P_j^{n-2j}(\Rm_k)\right|_k \\
  & \leq\sum_{p=0}^{n-2}a_p|\nabla_k^p(\Rm_k)|_k \\
  & \leq\sum_{p=0}^{n-2}a_pCC'_{p,0}\equiv\overline{C}_0.
 \end{align*}
Then
\begin{align*}
\overline{C}_0|t_1-t_0| & \geq\int_{t_0}^{t_1}\left|\frac{\dif}{\dif t}\log g_k(t)(V,V)\right|\,dt \\
& \geq\left|\int_{t_0}^{t_1}\frac{\dif}{\dif t}\log g_k(t)(V,V)\,dt\right| \\
& \geq\left|\log\frac{g_k(t_1)(V,V)}{g_k(t_0)(V,V)}\right|,
\end{align*}
 which yields
 \begin{gather*}
  e^{-\bar{C}_0|t_1-t_0|}g_k(t_0)(V,V)\leq g_k(t_1)(V,V)\leq e^{\bar{C}_0|t_1-t_0|}g_k(t_0)(V,V) \\
  B_0^{-1}e^{-\bar{C}_0|t_1-t_0|}g(V,V)\leq g_k(t_1)(V,V)\leq B_0e^{\bar{C}_0|t_1-t_0|}g(V,V).
 \end{gather*}
\end{proof}
We will need the following two lemmas in the next proof.
\begin{lemma} \label{lem: tensor norm equivalence}
 (Chow et al. \cite{ChowEtAlRicciFlowTechniquesApplications} Lemma 3.13). Suppose that the metrics $g$ and $h$ are equivalent: $C^{-1}g\leq h\leq Cg$. Then for any $(p,q)$-tensor $T$, we have $|T|_h\leq C^{(p+q)/2}|T|_g$.
\end{lemma}
\begin{lemma} \label{lem: equivalence of nabla(g_k) and Gamma_k-Gamma}
(Chow et al. \cite{ChowEtAlRicciFlowTechniquesApplications} Lemma 3.11)  Let $(M,g)$ be a Riemannian manifold, and let $\{g_k(t)\}_{k\in\N}$ be a collection of metrics on $M$. Then for each $k$, $\nabla g_k(t)$ and $\Gamma_k(t)-\Gamma$ are equivalent:
 \begin{equation*}
  \tfrac{1}{2}|\nabla g_k(t)|_k\leq |\Gamma_k(t)-\Gamma|_k\leq\tfrac{3}{2}|\nabla g_k(t)|_k.
 \end{equation*}
 \end{lemma}
\begin{lemma} \label{lem: extend convergence of metrics over time interval: ptwise estimates on spacetime derivatives of metrics}
 For every $p,q\geq 0$, there is a constant $\tilde{C}_{p,q}$ independent of $k$ such that $|\dif_t^q\nabla^p g_k(t)|\leq\tilde{C}_{p,q}$ on $L\times[\beta,\psi]$. 
\end{lemma}
\begin{proof}
Define the bounds $\overline{C}_j$ for $j$ satisfying $0\leq j\leq j-n+2$ by
\begin{equation*}
|\nabla_k^j\Obhat_k|\leq\sum_{p=j}^{n-2+j}a_pCC'_{p,0}\equiv\overline{C}_j.
\end{equation*}

We first prove the lemma for $(p,q)=(1,0)$. Hamilton showed in Theorem 7.1 of \cite{Hamilton3MfdPositiveRicciCurvature} that $\dif_t\Gamma=g^{-1}*\nabla\dif_t g$. Then
\begin{equation*}
 |\dif_t(\Gamma_k-\Gamma)|_k\leq C|\nabla_k\Obhat_k|_k\leq C\overline{C}_1.
\end{equation*}
So
\begin{align*}
 C\overline{C}_1|t_1-t_0| & \geq\int_{t_0}^{t_1}|\dif_t(\Gamma_k(t)-\Gamma)|_k\,dt \\
 & \geq\left|\int_{t_0}^{t_1}\partial_t(\Gamma_k(t)-\Gamma)\,dt\right|_k \\
 & \geq|\Gamma_k(t_1)-\Gamma|_k-|\Gamma_k(t_0)-\Gamma|_k.
\end{align*}
This gives
\begin{align*}
 |\Gamma_k(t)-\Gamma|_k & \leq C\overline{C}_1|t-t_0|+|\Gamma_k(t_0)-\Gamma|_k \\
 & \leq C\overline{C}_1|t-t_0|+\tfrac{3}{2}|\nabla g_k(t_0)|_k \\
 & \leq C\overline{C}_1|t-t_0|+\tfrac{3}{2}B_0^{3/2}C_1 \\
 & \leq C\overline{C}_1|\psi-\beta|+\tfrac{3}{2}B_0^{3/2}C_1.
\end{align*}
We used Lemma \ref{lem: equivalence of nabla(g_k) and Gamma_k-Gamma} in the second line and Lemma \ref{lem: tensor norm equivalence} in the third line. Then
\begin{align*}
 |\nabla g_k(t)| & \leq B(t,t_0)^{3/2}|\nabla g_k(t)|_k \\
 & \leq B(\psi,\beta)^{3/2}2|\Gamma_k(t)-\Gamma|_k \\
 & \leq B(\psi,\beta)^{3/2}(C\overline{C}_1|\psi-\beta|+3B_0^{3/2}C_1)\equiv\tilde{C}_{1,0},
\end{align*}
where we used Lemma \ref{lem: equivalence of nabla(g_k) and Gamma_k-Gamma} in the second line.

Next, we prove the lemma for $p$ satisfying $p\leq k$ when $q=0$. We will show that for $p\geq 1$,
\begin{equation} \label{eq: nabla^p partial(g_k)/partial(t) induction conclusion}
 |\nabla^p\dif_t g_k|\leq C''_p|\nabla^p g_k|+C'''_p,\quad |\nabla^p g_k|\leq\tilde{C}_{p,0}. 
\end{equation}

If $p=1$, then
\begin{align*}
 |\nabla\dif_tg_k(t)| & \leq B(t,t_0)^{3/2}|(\nabla-\nabla_k)\dif_tg_k+\nabla_k\dif_tg_k|_k \\
 & \leq B(t,t_0)^{3/2}C|\Gamma-\Gamma_k|_k|\dif_tg_k|_k+|\nabla_k\dif_tg_k|_k \\
 & \leq B(t,t_0)^{3/2}C|\nabla g_k|\overline{C}_0+\overline{C}_1
\end{align*}
and we have already shown that $|\nabla g_k|\leq\tilde{C}_{1,0}$.

Let $N\geq 2$ and assume that \eqref{eq: nabla^p partial(g_k)/partial(t) induction conclusion} is true for $0\leq p\leq N-1$. The telescoping identity
\begin{equation*}
\nabla^NA-\nabla_k^NA=\sum_{i=1}^N\nabla^{N-i}(\nabla-\nabla_k)\nabla_k^{i-1}A
\end{equation*}
gives
\begin{align} \label{eqn: nabla g_k differential inequality internal eqn q=0}
 |\nabla^N\dif_tg_k| & =\left|\sum_{i=1}^N\nabla^{N-i}(\nabla-\nabla_k)\nabla_k^{i-1}\dif_tg_k+\nabla_k^N\dif_tg_k\right| \nonumber \\
 & \leq\sum_{i=1}^N|\nabla^{N-i}(\nabla-\nabla_k)\nabla_k^{i-1}\dif_tg_k|+|\nabla_k^N\dif_tg_k| \nonumber \\
 & =|\nabla^{N-1}(\nabla-\nabla_k)\dif_tg_k|+\sum_{i=2}^N|\nabla^{N-i}(\nabla-\nabla_k)\nabla_k^{i-1}\dif_tg_k|+|\nabla_k^N\dif_tg_k|. 
\end{align}
We estimate the first term of \eqref{eqn: nabla g_k differential inequality internal eqn q=0}:
\begin{align*}
|\nabla^{N-1}(\nabla-\nabla_k)\dif_tg_k| & =|\nabla^{N-1}(\nabla g_k*\dif_tg_k)| \\
& \leq\sum_{j=0}^{N-1}b_j|\nabla^{N-j}g_k||\nabla^j\dif_tg_k| \\
& \leq b_0\overline{C}_0|\nabla^N g_k|+\sum_{j=1}^{N-1}b_j\tilde{C}_{N-j,0}(C''_j\tilde{C}_{j,0}+C'''_j).
\end{align*}
The first equality is due to the identity
\begin{equation} \label{eqn:  nabla g_k differential inequality gamma difference equiv to nabla}
(g_k)^{ec}(\nabla_a(g_k)_{bc}+\nabla_b(g_k)_{ac}-\nabla_c(g_k)_{ab})=2(\Gamma_k)_{ab}^e-2\Gamma_{ab}^e.
\end{equation}
We estimate the second term of \eqref{eqn: nabla g_k differential inequality internal eqn q=0}:
\begin{align*}
\sum_{i=2}^N|\nabla^{N-i}(\nabla-\nabla_k)\nabla_k^{i-1}\dif_tg_k|
& =\sum_{i=2}^N|\nabla^{N-i}(\nabla g_k*\nabla_k^{i-1}\dif_tg_k)| \\
& \leq\sum_{i=2}^N\sum_{j=0}^{N-i}b'_j|\nabla^{N-i-j+1}g_k||\nabla^j\nabla_k^{i-1}\dif_tg_k| \\
&  \leq\sum_{i=2}^N\sum_{j=0}^{N-i}b'_j|\nabla^{N-i-j+1}g_k|\sum_{l=0}^j|(\nabla-\nabla_k)^l\nabla_k^{j-l}\partial_tg_k| \\
&  =\sum_{i=2}^N\sum_{j=0}^{N-i}b'_j|\nabla^{N-i-j+1}g_k|\sum_{l=0}^j|(\nabla g_k)^{*l}*\nabla_k^{j-l+i-1}\partial_t g_k| \\
&  \leq\sum_{i=2}^N\sum_{j=0}^{N-i}b'_j|\nabla^{N-i-j+1}g_k|\sum_{l=0}^j b''_l|\nabla g_k|^l|\nabla_k^{j-l+i-1}\partial_t g_k| \\
& \leq\sum_{i=2}^N\sum_{j=0}^{N-i}\sum_{l=0}^j b'_j\tilde{C}_{N-i+j-1,0}b''_l\tilde{C}_{1,0}^l\overline{C}_{j-l+i-1}.
\end{align*}
We applied \eqref{eqn:  nabla g_k differential inequality gamma difference equiv to nabla} in the first and fourth lines. The last term of \eqref{eqn: nabla g_k differential inequality internal eqn q=0} is also bounded: $|\nabla_k^N\dif_tg_k|\leq\overline{C}_N$. Collecting the three previous estimates, we obtain
\begin{equation*}
|\nabla^N\dif_tg_k|\leq C''_N|\nabla^N g_k|+C'''_N.
\end{equation*}
Applying the preceding inequality, we get
\begin{align} \label{eq: cptness induction diff ineq}
 \dif_t|\nabla^Ng_k|^2 & =2\langle\dif_t\nabla^Ng_k,\nabla^Ng_k\rangle \nonumber \\
 & \leq |\dif_t\nabla^Ng_k|^2+|\nabla^Ng_k|^2 \nonumber \\
 & \leq (1+2(C''_N)^2)|\nabla^Ng_k|^2+2(C'''_N)^2.
\end{align}
The solution of the ODE $dA/d\sigma=c_1A+c_2$ is given by
\begin{equation} \label{eq: cptness induction ODE}
 A(\sigma)=e^{c_1(\sigma-\sigma_1)}\left[A(\sigma_1)+\tfrac{c_2}{c_1}(1-e^{-c_1(\sigma-\sigma_1)})\right].
\end{equation}
Applying \eqref{eq: cptness induction ODE} to \eqref{eq: cptness induction diff ineq}, we get
\begin{align*}
 |\nabla^Ng_k|^2(t)
 & \leq e^{(1+2(C''_N)^2)(t-t_0)}\left[|\nabla^Ng_k|^2(t_0)+\frac{2(C'''_N)^2}{1+2(C''_N)^2}\big(1-e^{(1+2(C''_N)^2)(t_0-t)}\big)\right] \\
 & \leq e^{(1+2(C''_N)^2)(\psi-t_0)}\left[C_N+\frac{2(C'''_N)^2}{1+2(C''_N)^2}\big(1-e^{(1+2(C''_N)^2)(t_0-\beta)}\big)\right]\equiv\tilde{C}_{N,0}^2.
\end{align*}
This completes the inductive proof of \eqref{eq: nabla^p partial(g_k)/partial(t) induction conclusion} and the proof of the lemma for any $p$ when $q=0$.
Since $\dif_t^q\nabla^pg_k=\nabla^p\dif_t^qg_k$, a similar proceedure may be used to prove the lemma when $q>0$.
\end{proof}
\begin{proof}[Proof of Proposition \ref{prop: extend convergence of metrics over time interval}]
 The two preceding Lemmas \ref{lem: extend convergence of metrics over time interval: uniform spacetime equivalence of metrics} and \ref{lem: extend convergence of metrics over time interval: ptwise estimates on spacetime derivatives of metrics} prove the proposition.
\end{proof}
We are now able to prove the compactness Theorem \ref{thm: cptness for seq of cpt mfds} for solutions of the AOF. We need the following lemma.
\begin{proposition} \label{prop: arzela-ascoli for metrics}
 (Chow et al. \cite{ChowEtAlRicciFlowTechniquesApplications} Corollary 3.15) Let $(M^n,g)$ be a Riemannian manifold and let $L\subset M^n$ be compact. Furthermore, let $p$ be a nonnegative integer. If $\{g_k\}_{k\in\N}$ is a sequence of Riemannian metrics on $L$ such that
 \begin{equation*}
  \sup_{0\leq|\alpha|\leq p+1}\sup_{x\in L}|\nabla^\alpha g_k|\leq C<\infty
 \end{equation*}
 and if there exists $\delta>0$ such that $g_k(V,V)\geq\delta g(V,V)$ for all $V\in TM$,then there exists a subsequence $\{g_k\}$ and a Riemannian metric $g_\infty$ on $L$ such that $g_k$ converges in $C^p$ to $g_\infty$ as $k\to\infty$.
\end{proposition}
\begin{proof}[Proof of Theorem \ref{thm: cptness for seq of cpt mfds}]
Since we are given a uniform bound on $|\Rm(g_k)|_{g_k}$, the pointwise smoothing estimates given by Theorem \ref{thm:C^m smoothing estimate along AOF} furnish unform bounds on $\|\nabla_{g_k(t_0)}^m\Rm(g_k(t_0))\|_{C^0(g_k(t_0))}$ for all $m\in\N$. Therefore, since the $(M_k,g_k)$ are complete, the Cheeger-Gromov compactness Theorem \ref{thm: Cheeger Gromov cptness} yields a subsequence of $\{(M_k,g_k(t),O_k)\}_{k\in\N}$, also called $\{(M_k,g_k(t),O_k)\}_{k\in\N}$, for which $\{(M_k,g_k(t_0),O_k)\}_{k\in\N}$ converges to a complete pointed Riemannian manifold $(M_\infty^n,h,O_\infty)$.

 Fix a compact subset $L$ of $M_\infty$ and a closed interval $[\beta,\psi]$, with $t_0\in(\beta,\psi)$ of $(\alpha,\omega)$. Since $\{(M_k,g_k(t_0),O_k)\}_{k\in\N}$ converges to $(M_\infty,h,O_\infty)$, by definition there exists an exhaustion $\{U_k\}_{k\in\N}$ of $M_\infty$ by open sets with $O_\infty\in U_k$ and a sequence of diffeomorphisms $\Phi_k:U_k\to V_k\equiv\Phi_k(U_k)\subset M_k$ with $\Phi_k(O_\infty)=O_k$, such that if $h_k\equiv\Phi_k^*\left[g_k(t_0)|_{V_k}\right]$ , then $(U_k,h_k)$ converges in $C^\infty$ to $(M_\infty,h)$ on compact sets in $M_\infty$. Since the $U_k$ exhaust $M_\infty$, $L\subset U_k$ for some $k$. So the metrics $h_k$ are uniformly equivalent to $h$ on $L$. We also obtain from the $C^\infty$ convergence that for each $p\geq 1$, there exists a $C_p$ independent of $x\in L$ and $k$ such that $|\nabla_h^p h_k|_h\leq C_p$.

Let $G_k(t)\equiv\Phi_k^*\left[g_k(t)|_{V_k}\right]$; then $h_k=G_k(t_0)$. From  the pointwise smoothing estimates given by Theorem \ref{thm:C^m smoothing estimate along AOF}, for each $p$ we obtain a bound $C'_{p,0}$ uniform on $L\x[\beta,\psi]$ independent of $k$ such that $|\nabla_{G_k}^p\Rm(G_k)|_{G_k}\leq C'_{p,0}$ on $L\times[\beta,\psi]$. Using the expression of $\dif_t\nabla_{G_k}^p\Rm(G_k)$ in terms of covariant derivatives of $\Rm(G_k)$ given by Proposition \ref{prop: evolution of nabla^k Rm}, for each $(p,q)$ we obtain a bound $C'_{p,q}$ uniform on $L\x[\beta,\psi]$ independent of $k$ such that $|\dif_t^q\nabla_{G_k}^p\Rm(G_k)|_{G_k}\leq C'_{p,q}$ on $L\times[\beta,\psi]$. We then conclude via Proposition \ref{prop: extend convergence of metrics over time interval} that the metrics $G_k$ are uniformly equivalent to $h$ on $L\x[\beta,\psi]$ and that for every $p,q\geq 0$, there is a constant $\tilde{C}_{p,q}$ independent of $k$ such that $|\dif_t^q\nabla_h^p G_k|_h\leq\tilde{C}_{p,q}$ on $L\times[\beta,\psi]$.

The uniform equivalence of the $G_k$ to $h$ and the uniform bounds $|\dif_t^q\nabla_h^p G_k|_h\leq\tilde{C}_{p,q}$ allow us to apply an Arzel\`{a}-Ascoli type Proposition \ref{prop: arzela-ascoli for metrics} to the metrics $G_k(t)+dt^2$ on $L\times[\beta,\psi]$ and obtain a subsequence that converges in $C^\infty(L\times[\beta,\psi],h+dt^2)$ to a metric $g_\infty(t)+dt^2$ such that $g_\infty(0)=h$; we relabel the convergent subsequence as $\{G_k(t)+dt^2\}_{k\in\N}$. It follows that $g_\infty(t)+dt^2$ is uniformly equivalent to $G_1(t)+dt^2$  on $L\times[\beta,\psi]$. Then $g_\infty(t)+dt^2$ is uniformly equivalent to $h+dt^2$  on $L\times[\beta,\psi]$ since $G_1(t)+dt^2$ is uniformly equivalent to $h+dt^2$ on $L\times[\beta,\psi]$. Since $(M,h)$ is complete, $(M\x(\alpha,\omega),h+dt^2)$ is also complete. The uniform equivalence of $g_\infty(t)+dt^2$ to $h+dt^2$ on compact subsets of $M\x(\alpha,\omega)$ and the Hopf-Rinow theorem imply that $(M_\infty\x(\alpha,\omega),g_\infty(t)+dt^2)$ is complete.

Since $(M_\infty\x(\alpha,\omega),g_\infty(t)+dt^2)$ is complete, compact sets are equivalent to closed, bounded ones. A compact set in $M_\infty\x(\alpha,\omega)$ is contained in the compact set that is the product of a closed geodesic ball in $M_\infty$ and a closed interval in $(\alpha,\omega)$. So the metrics $G_k(t)+dt^2$ subsequentially converge in $C^\infty(M_\infty\x(\alpha,\omega),h+dt^2)$. Let $\{G_k(t)+dt^2\}_{k\in\N}$ be the convergent subsequence. Then $\{(M_k,g_k(t),O_k)\}_{k\in\N}$ converges to $(M_\infty,g_\infty(t),O_\infty)$. It follows that for each $p,q$, $\partial_t^q\nabla_{g_\infty}^p G_k\to\partial_t^q\nabla_{g_\infty}^p g_\infty $ and $\Obhat(G_k)\to\Obhat(g_\infty)$ in $C(M_\infty\x(\alpha,\omega),g_\infty(t)+dt^2)$. Therefore  $(M_\infty,g_\infty,O_\infty)$ is a complete pointed solution to AOF for $t\in(\alpha,\omega)$.
\end{proof}
As our first corollary of the compactness theorem \ref{thm: cptness for seq of cpt mfds}, we show that under suitable conditions, we can obtain a singularity model for the ambient obstruction flow.
\begin{proof}[Proof of Theorem  \ref{thm: existence of singularity model at curvature singularity}]
We first show that the $g_i$ are also solutions to AOF by showing that if $\tilde{g}=\lambda g$ and $g$ satisfies AOF, given up to constants by
 \begin{equation*}
  \dif_t g=\Delta^{\frac{n}{2}-1}\Rc+\Delta^{\frac{n}{2}-2}\nabla^2 R+\sum_{j=2}^{n/2}P_j^{n-2j}(\Rm),
 \end{equation*}
then $\tilde{g}$ satisfies
 \begin{equation}\label{eq: tilde PDE}
  \dif_t\tilde{g}=\widetilde{\Delta}^{\frac{n}{2}-1}\widetilde{\Rc}+\widetilde{\Delta}^{\frac{n}{2}-2}\widetilde{\nabla}^2\tilde{R}+\sum_{j=2}^{n/2}P_j^{n-2j}(\widetilde{\Rm}).
 \end{equation}
 We evaluate the first term of the right side of \eqref{eq: tilde PDE}:
\begin{equation*}
\widetilde{\Delta}^{\frac{n}{2}-1}\widetilde{\Rc}=\left(\lambda^{-1}g^{-1}\nabla^2\right)^{\frac{n}{2}-1}\Rc=\lambda^{1-\frac{n}{2}}\Delta^{\frac{n}{2}-1}\Rc.
\end{equation*}
Similarly, the second term is equal to $\lambda^{1-\frac{n}{2}}\Delta^{\frac{n}{2}-1}\Rc$. The remaining terms are contractions of terms of the form
\begin{equation*}
\widetilde{\nabla}^{i_1}\widetilde{\Rm}\otimes\cdots\otimes\widetilde{\nabla}^{i_j}\widetilde{\Rm}
\end{equation*}
with $2\leq j\leq\frac{n}{2}$ and $i_1+\cdots+i_j=n-2j$. In order to contract on all but two indices of the above term, we need to contract $\frac{1}{2}(i_1+\cdots i_j+3j-j-2)=\frac{n}{2}-1$ pairs of indices. This implies that $P_j^{n-2j}(\widetilde{\Rm})=\lambda^{1-\frac{n}{2}}P_j^{n-2j}(\Rm)$. The left side of \eqref{eq: tilde PDE} is equal to $\lambda^{1-\frac{n}{2}}\dif_t g$. So $\tilde{g}$ satisfies \eqref{eq: tilde PDE}.

We have $|\Rm(g_i)|_{g_i}\leq 1$ on $M\x[-\lambda_i^{n/2}t_i,0]$ for each $i$ since the definition of the $\lambda_i$ implies
\begin{equation*}
|\Rm(g_i)|_{g_i}^2=\lambda_i^{-2}|\Rm|^2\leq\lambda_i^{-2}\lambda_i^2\leq 1.
\end{equation*}
Let $k\in\N$. There exists $i_k$ such that if $i\geq i_k$, then $\lambda_i^{n/2}t_i>k$. Then $\{g_i\}_{i\geq i_k}$ is a sequence of complete pointed solutions to AOF on $(-k,0]$. Since the Sobolev constant is scaling invariant, the uniform bound of $C_S(M,g)$ on $[0,T)$ implies a uniform bound independent of $i$ of $C_S(M,g_i)$ on $[0,T)$. We conclude from Lemma 3.2 of Hebey \cite{HebeySobolevSpacesOnRiemannianManifolds} that there exists a uniform lower bound independent of $i$ for $\inf_{x\in M}\vol(B_{g_i}(x,1))$. This and the bound $|\Rm(g_i)|_{g_i}\leq 1$ on $M\x[-\lambda_i^{n/2}t_i,0]$ for all $i$ give  a uniform lower bound independent of $i$ for $\injrad_{g_i(0)}(x_i)$ via the Cheeger-Gromov-Taylor theorem.

The proof of the compactness theorem \ref{thm: cptness for seq of cpt mfds} is unchanged if we replace $(\alpha,\omega)$ with $(-k,0]$. Thus, by theorem \ref{thm: cptness for seq of cpt mfds}, we obtain subsequential convergence of $\{(M,g_i(t),x_i)\}_{i\geq i_k}$ to a complete pointed solution $(M_\infty,g_\infty(t),x_\infty)$ to AOF for $t\in(-k,0]$.  By taking a further diagonal subsequence over the $k$, we get that $\{(M,g_i(t),x_i)\}_{i\geq 1}$ subsequentially converges to a complete pointed solution $(M_\infty,g_\infty(t),x_\infty)$ to AOF for $t\in(-\infty,0]$. The limit $(M_\infty,g_\infty(t))$ is not flat since
\begin{equation*}
|\Rm(g_\infty(0))(x_\infty)|_{g_\infty(0)}=1
\end{equation*}
by the definition of the $g_i(t)$.

We show that $M_\infty$ is not compact. Lemma 3.9 of Chow-Knopf \cite{ChowKnopfRicciFlowIntro} states that for a one parameter family of Riemannian manifolds $(M,g(t))$, the volume element evolves by $\dif_t dV_g=\frac{1}{2}g^{ij}\dif_t g_{ij}$. Applying the fact that $\Ob$ is traceless and the divergence theorem,
\begin{align*}
\frac{\dif}{\dif t}\vol(M,g(t)) & =\frac{1}{2}\int_M g^{ij}\frac{\dif g_{ij}}{\dif t}\, dV_{g(t)} \\
& =\frac{1}{2}\int_M[(-1)^\frac{n}{2}g^{ij}\Ob_{ij}+C(n)(\Delta^{\frac{n}{2}-1}R)g^{ij}g_{ij}]\,dV_{g(t)} \\
& =C(n)\int_M\Delta^{\frac{n}{2}-1}R\,dV_{g(t)} \\
& =0.
\end{align*}
Therefore the volume of $(M,g(t))$ is preserved along the flow. Since $\lambda_i\to\infty$,
\begin{equation*}
\vol(M_\infty,g_\infty(t))=\lim_{i\to\infty}\vol(M,g_i(t))=\lim_{i\to\infty}\lambda_i^{n/2}\vol(M,g(t_i+\lambda_i^\frac{n}{2}t))=\infty
\end{equation*}
for all  $t\in(-\infty,0]$. So the volume of $(M,g_\infty(t))$ is infinite for all $t\in(-\infty,0]$. The uniform volume lower bound for the $(M,g_i)$ passes in the limit to a uniform volume lower bound for $(M,g_\infty)$. Therefore $M_\infty$ is noncompact by Lemma 8.1 of Bour \cite{BourFourthOrderCurvatureFlowsAndGeometricApplications}.

Next, we show that the integral of the $Q$-curvature is nondecreasing along the flow on $M$. Along the flow, the derivative of $\int_M Q$ is given by
\begin{align*}
\frac{\partial}{\partial t}\int_M Q & =(-1)^\frac{n}{2}\frac{n-2}{2}\int_M \langle\Ob,\partial_t g\rangle \\
& =(-1)^\frac{n}{2}\frac{n-2}{2}\int_M (-1)^\frac{n}{2}|\Ob|^2+C(n)\int_M \langle\Ob,(\Delta^{\frac{n}{2}-1}R)g\rangle \\
& =\frac{n-2}{2}\int_M |\Ob|^2,
\end{align*}
where the third line holds since $\Ob$ is traceless. So the integral of the $Q$-curvature does not decrease along the flow.

Suppose that
\begin{equation*}
\sup_{t\in[0,T)}\int_M Q(g(t))\,dV_{g(t)}<\infty.
\end{equation*}
This is always true when $n=4$ since the Chern-Gauss-Bonnet theorem gives that for all $t\in[0,T)$,
\begin{equation*}
\int_M Q =8\pi^2\chi(M)-\tfrac{1}{4}\int_M|W|^2\leq 8\pi^2\chi(M).
\end{equation*}
So if the integral of the $Q$ curvature is bounded along the flow,
\begin{align*}
\int_0^T\int_M|\Ob|^2 & =\int_0^T\frac{\partial}{\partial t}\int_M Q \\
& =\lim_{t\uparrow T}\int_M Q(g(t))-\int_M Q(g(0)) \\
& <\infty.
\end{align*}
Let $\{(M,g_i(t),x_i)\}_{i\geq 1}$ be the convergent subsequence previously found in the proof. Fix $k\in\N$. Since $t_i\to T$ and $\lambda_i\to\infty$, we can choose a subsequence of times $\{t_{i_j}\}_{j\in\N}$ as follows:
\begin{equation*}
 i_1=\inf\big\{i:t_i\geq\tfrac{T}{2}, \lambda_i\geq\big(\tfrac{2k}{T}\big)^\frac{2}{n}\big\}, \quad i_j=\inf\big\{i:t_i\geq\tfrac{1}{2}(T+t_{i_{j-1}}),\lambda_i\geq\big(\tfrac{2k}{T-t_{i_{j-1}}}\big)^\frac{2}{n}\big\}
\end{equation*}
for $j\geq 2$. We relabel $\{t_{i_j}\}_{j\in\N}$ as$\{t_i\}_{i\in\N}$. Then
\begin{equation*}
\sum_{i=1}^\infty\int_{t_i-k\lambda_i^{-\frac{n}{2}}}^{t_i}\int_M|\Ob|^2<\int_0^T\int_M|\Ob|^2<\infty,
\end{equation*}
implying that, using the scaling law $\Ob(\lambda g)=\lambda^{\frac{2-n}{2}}\Ob(g)$,
\begin{align*}
0 & =\lim_{i\to\infty}\int_{t_i-k\lambda_i^{-\frac{n}{2}}}^{t_i}\int_M|\Ob(g)|_{g}^2\,dV_g\,dt \\
& =\lim_{i\to\infty}\int_{-k}^0\int_M\lambda_i^n|\Ob(g_i)|_{g_i}^2\lambda_i^{-\frac{n}{2}}\lambda_i^{-\frac{n}{2}}\,dV_{g_i}\,dt \\
& =\lim_{i\to\infty}\int_{-k}^0\int_M|\Ob(g_i)|_{g_i}^2\,dV_{g_i}\,dt.
\end{align*}
Since $\Ob(g_i)\to\Ob(g_\infty)$ in $C^\infty$ on compact subsets, this implies that $\Ob(g_\infty)\equiv 0$ on $[-k,0]$. So for each $k\in\N$, there exists a sequence of pointed solutions to AOF that converge to an obstruction flat pointed solution to AOF on $[-k,0]$. By taking a further diagonal subsequence over the $k$, we obtain a sequence of pointed solutions to AOF that converge to an obstruction flat complete pointed solution to AOF on $(-\infty,0]$.
\end{proof}
Finally, we provide a corollary of the compactness theorem \ref{thm: cptness for seq of cpt mfds} characterizing limits of nonsingular solutions to AOF.
\begin{proof}[Proof of Theorem \ref{thm: nonsingular limit at infinity}]
Suppose $M$ does not collapse at $\infty$. Then there exists a sequence $\{(x_i,t_i)\}_{i\in\N}\subset M\x[0,\infty)$ such that $\inf_i\injrad_{g(t_i)}(x_i)>0$. Let $g_i(t)=g(t+t_i)$ for $t\in[-t_i,\infty)$. Let $k\in\N$. Then there exists $i_k\in\N$ such that $t_i> k$ for all $i\geq i_k$. Since $\sup_{t\in[0,\infty)}\|\Rm\|_\infty<\infty$ and $\inf_i\injrad_{g_i(0)}(x_i)>0$, we apply Theorem \ref{thm: cptness for seq of cpt mfds} to obtain subseqential convergence in the sense of families of pointed Riemannian manifolds of $\{(M,g_i(t),x_i)\}_{i\geq i_k}$ to a complete pointed solution $(M_\infty,g_\infty(t),x_\infty)$ to AOF on $(-k,\infty)$. By taking a further diagonal subsequence over the $k$, we get that $\{(M,g_i(t),x_i)\}_{i\geq 1}$ subsequentially converges to a complete pointed solution $(M_\infty,g_\infty(t),x_\infty)$ to AOF on $(-\infty,\infty)$.

If $M_\infty$ is compact, then by the definition of convergence of complete pointed Riemannian manifolds, $M_\infty$ is diffeomorphic to $M$. Just as in the proof of Theorem \ref{thm: existence of singularity model at curvature singularity}, the volume of $(M,g(t))$ is preserved along the flow. So for all $t\in(-\infty,\infty)$,
\begin{equation*}
\vol(M_\infty,g_\infty(t))=\lim_{i\to\infty}\vol(M,g_i(t))=\lim_{i\to\infty}\vol(M,g(t_i+t))<\infty.
\end{equation*}
Suppose that
\begin{equation*}
\sup_{t\in[0,\infty)}\int_M Q(g(t))\,dV_{g(t)}<\infty.
\end{equation*}
This is always true when $n=4$ by the Chern-Gauss-Bonnet theorem. Using the same argument as in the proof of Theorem \ref{thm: existence of singularity model at curvature singularity}, we obtain
\begin{equation*}
\int_0^\infty\int_M|\Ob|^2<\infty.
\end{equation*}
Let $\{(M,g_i(t),x_i)\}_{i\geq 1}$ be the convergent subsequence previously found in the proof. Since $t_i\to\infty$, we can choose a subsequence of times $\{t_{i_j}\}_{j\in\N}$ as follows:
\begin{equation*}
i_1=\inf\{i:t_i\geq k\}, \quad i_j=\inf\{i:t_i\geq t_{i_{j-1}}+2k\}
\end{equation*}
for $j\geq 2$. We relabel $\{t_{i_j}\}_{j\in\N}$ as$\{t_i\}_{i\in\N}$. Then
\begin{equation*}
\sum_{i=1}^\infty\int_{t_i-k}^{t_i+k}\int_M|\Ob|^2<\int_0^\infty\int_M|\Ob|^2<\infty
\end{equation*}
implies that
\begin{equation*}
0 =\lim_{i\to\infty}\int_{t_i-k}^{t_i+k}\int_M|\Ob(g)|_g^2\,dV_g\,dt  =\lim_{i\to\infty}\int_{-k}^{k}\int_M|\Ob(g_i)|_{g_i}^2\,dV_{g_i}\,dt.
\end{equation*}
Since $\Ob(g_i)\to\Ob(g_\infty)$ in $C^\infty$ on compact subsets, this implies that $\Ob(g_\infty)\equiv 0$ on $[-k,k]$. So for each $k\in\N$, there exists a sequence of pointed solutions to AOF that converge to an obstruction flat pointed solution to AOF on $[-k,k]$. By taking a further diagonal subsequence over the $k$, we obtain a sequence of pointed solutions to AOF that converge to an obstruction flat complete pointed solution to AOF on $(-\infty,\infty)$. Since $g_\infty$ solves the conformal flow $\dif_t g_\infty=(-1)^{n/2}C(n)(\Delta^{\frac{n}{2}-1}R)g$, we see that $g_\infty(t)$ is in the conformal class of $g_\infty(0)$ for all $t\in(-\infty,\infty)$. If $M_\infty$ is compact, we can solve the Yamabe problem for $(M_\infty,[g_\infty(0)])$; the Yamabe problem was solved by Aubin, Trudinger, and Schoen (see \cite{AubinEquationsDifferentiellesNonLineairesEtProblemeDeYamabeConcernantLaCourbureScalaire, LeeParkerYamabeProblem}). Due to the conformal covariance of $\Ob$, we obtain a obstruction flat, constant scalar curvature complete pointed solution $(M_\infty,\hat{g}_\infty(t))$ to AOF with $\hat{g}_\infty(t)=\hat{g}_\infty(0)$ for all $t\in(-\infty,\infty)$.
\end{proof}
\bibliographystyle{hamsplain}

%%%%%%%%%%%%%%%%%%%%%%%%%%%

%
%
%
%
%
%
%
%
%
%
%
%
%
%
%
%
%
%
%
%
%
%
%
%
%
%
%
%
%
%
%
%
%
%
\end{document}